\documentclass[11pt]{amsart}

\usepackage[T1]{fontenc}
\usepackage{palatino} 
\usepackage{tikz}
\usetikzlibrary{arrows.meta, decorations.markings}
\usepackage[most]{tcolorbox}

\usepackage[a4paper,margin=2.5cm]{geometry}

\usepackage{amsmath,amssymb}

\usepackage{xcolor}
\usepackage[bookmarks=true]{hyperref}
\usepackage[nameinlink]{cleveref}
\newcommand{\orcidlink}[1]{\href{https://orcid.org/#1}{\texttt{#1}}}

\theoremstyle{plain}
\newtheorem{theorem}{Theorem}[section]
\newtheorem*{theorem*}{Theorem}
\newtheorem{proposition}[theorem]{Proposition}
\newtheorem{corollary}[theorem]{Corollary}
\newtheorem{lemma}[theorem]{Lemma}

\theoremstyle{definition}
\newtheorem{definition}[theorem]{Definition}
\newtheorem{notation}[theorem]{Notation}

\theoremstyle{remark}
\newtheorem{remark}[theorem]{Remark}
\newtheorem{example}[theorem]{Example}

\renewcommand{\bar}[1]{\overline{#1}} 

\DeclareMathOperator{\Tr}{Tr}
\DeclareMathOperator{\Aut}{Aut}
\DeclareMathOperator{\cdiv}{cdiv}

\DeclareMathOperator{\ord}{ord}

\DeclareMathOperator{\Span}{span}

\newcommand{\CC}{\mathbb{C}}
\newcommand{\RR}{\mathbb{R}}
\newcommand{\ZZ}{\mathbb{Z}}
\newcommand{\NN}{\mathbb{N}}
\newcommand{\HH}{\mathbb{H}}
\newcommand{\SF}{\mathbb{S}}
\newcommand{\OO}{\mathbb{O}}

\newcommand{\Dd}{\mathcal{D}}
\newcommand{\Ss}{\mathcal{S}}
\newcommand{\Mm}{\mathcal{M}}

\newcommand{\Zz}{\mathcal{Z}}
\newcommand{\Ii}{\mathcal{I}}

\newcommand{\ii}{\sqrt{-1}} 
\newcommand{\C}{\CC}           

\begin{document}
\setlength{\parskip}{1pt}

\title{(Semi-)Models for Slice Regular Functions on Real Division Algebras}

\author[A. Altavilla]{Amedeo Altavilla}
\address{Dipartimento di Matematica, Universit\`a degli Studi di Bari Aldo Moro,
Via E. Orabona 4, 70125 Bari, Italy}
\email{amedeo.altavilla@uniba.it}
\urladdr{ORCID: \orcidlink{0000-0002-3290-371X}}
\thanks{A.\ Altavilla was partially supported by PRIN 2022MWPMAB -- ``Interactions
between Geometric Structures and Function Theories'' and by GNSAGA of INdAM.}

\author[C. Bisi]{Cinzia Bisi}
\address{Dipartimento di Matematica e Informatica, Universit\`a di Ferrara,
Via N. Machiavelli 30, 44121 Ferrara, Italy}
\email{bsicnz@unife.it}
\urladdr{ORCID: \orcidlink{0000-0002-4973-1053}}
\thanks{C.\ Bisi was partially supported by PRIN 2022 ``Variet\`a reali e complesse:
geometria, topologia e analisi armonica'' and by GNSAGA of INdAM.}

\date{\today}

\begin{abstract}
We study slice regular functions on the real division algebras $\HH$ and
$\OO$ from the point of view of equivalence under automorphism and conjugation
actions. Motivated by recent orbit-theoretic descriptions in terms of the
invariants \((\Tr,N,\cdiv)\), and by the quaternionic theory of
$*$-conjugation by semiregular functions, we investigate whether a slice regular
function can be replaced by a canonical representative in its equivalence class.

The normalization considered in this paper is algebraic: we
look for representatives whose isolated nonreal zeros are aligned in a single
complex slice. We call such a representative a model when it belongs to the
strong-equivalence class of the original function. We prove that purely
vectorial functions always admit models, while in general strong equivalence is
too rigid and such representatives need not exist.

We therefore introduce semi-models, which preserve the symmetrization and the
real and spherical part of the zero set, but are allowed to leave the original
strong-equivalence class. Our main result proves that every slice regular
function \(f\) with \(N(f)\not\equiv 0\) admits a semi-model. The construction
clarifies the different roles of the invariants \((\Tr,N,\cdiv)\) and the
distinction between strong equivalence, weak equivalence, and semiregular
conjugacy.
\end{abstract}

\maketitle

\setcounter{tocdepth}{1} 
\section{Introduction}

Since its introduction \cite{gentili2006new,gentili2007new}, the theory of slice regular functions has provided a natural extension of one-variable holomorphic function theory to the quaternionic algebra $\HH$ and, more generally, to suitable hypercomplex algebras \cite{colombo2008slice,gentilistruppaCayley,ghiloni2011slice}. In this paper we work with the two real division algebras
\[
A\in\{\HH,\OO\},
\]
namely the quaternions and the octonions. Our focus is on the zero sets of slice regular functions and, more precisely, on the extent to which isolated nonreal zeros can be ``aligned'' inside a single complex slice.

One of the original motivations of the theory was to build a function theory in which polynomials and power series of the form
\[
\sum_{n\ge 0}x^n a_n
\]
are regular, while retaining a sufficiently rich analytic structure, encoded for instance by the Representation Formula and by the slice-by-slice holomorphic behavior \cite{GSSbook}. It soon became clear that slice regularity goes substantially beyond this first class of examples, and that one still has meaningful notions of analyticity, such as power and spherical analyticity \cite{gentili2008power,stoppato2012new,ghiloni2014power,altavilla2021spherical}, and more recently \cite{BisiCarbone}. At the algebraic level, the natural operation is not composition, which in general does not preserve slice regularity, but rather the $*$-product, induced by pointwise multiplication of the corresponding stem functions in the complexified algebra $A_\CC$ \cite{ghiloni2011slice}.

For the purposes of the present paper it is especially convenient to adopt the
``complex-curve'' viewpoint already implicit in
\cite[Remarks 3(2)]{ghiloni2011slice} and later developed explicitly in
\cite[Proposition 3.12]{Colombo2013} and \cite[Lemma 6.11]{GMP}.
Namely, once an orthonormal basis of $A$ is fixed, a slice regular function can
be encoded by a tuple of slice-preserving components, equivalently by an
intrinsic holomorphic curve with values in
$A_\CC\simeq\CC^{\dim_\RR A}$. From this point of view, the $*$-product
becomes pointwise multiplication in the complexified algebra, and the action
of holomorphic families of automorphisms of $A_\CC$ becomes a natural tool to
compare slice regular functions.

This complex-curve viewpoint is also motivated by recent geometric
applications of quaternionic and complex-quaternionic methods. In
\cite{AltavillaSchroeckerSirVrsekMinimalSurfaces}, isothermal minimal surfaces
are represented by holomorphic curves with values in the algebra of complex
quaternions, and their construction is reduced to a conjugation problem for
null complex-quaternionic vectors. This provides one geometric motivation for
the present paper: we investigate, in the slice regular setting and also over
the octonions, how far a function can be normalized within its automorphism or
conjugation class. A related rigidity phenomenon appears in
\cite{AltavillaSchroeckerSirVrsekPHPreserving}, where PH-preserving mappings
are characterized as conformal maps with square rational dilation. The same
general philosophy is also connected with the twistor-geometric approach to
slice regularity: in the quaternionic case, the twistor transform realizes a
slice regular function as a holomorphic curve in the Grassmannian
\(\operatorname{Gr}_2(\CC^4)\), and recent work of Liu, Moreno and Shi uses
this construction to study equivalence classes of slice regular functions and
slice regular polynomials under natural projective group actions
\cite{LiuMorenoShi2026}.

This perspective lies behind a family of recent equivalence theories for slice regular functions \cite{AltavillaLAA,Bisi2025,Bisioctonions}. In the quaternionic and octonionic settings, one finds that the relevant invariants are the trace, the symmetrization, and the so-called central divisor. More precisely, in the works of Bisi and Winkelmann on the orbit problem for holomorphic curves associated with slice regular functions \cite{Bisi2025,Bisioctonions}, if two functions share the three invariants
\[
(\Tr,N,\cdiv),
\]
then they belong to the same holomorphic orbit; equivalently, there exists a global holomorphic section of the corresponding orbit bundle over the complex base. In the quaternionic case this orbit picture can be interpreted locally in terms of zero-free regular conjugators. On the other hand, the quaternionic semiregular-conjugacy viewpoint developed in \cite{AltavillaLAA} shows a genuinely different phenomenon: if one prescribes only the two invariants
\[
(\Tr,N),
\]
then one still obtains $*$-conjugacy at the semiregular level, but in general one loses the existence of a global nowhere vanishing regular conjugator. The mismatch of the central divisor is precisely the obstruction. These invariants also arise naturally in related questions, for instance in problems concerning logarithms and factorization phenomena \cite{AltavillaLINCEI,AltavillaMongodi,GPV,GPV2}.

This naturally suggests a comparison with complex analysis. In one and several complex variables, a fruitful model theory has been developed for holomorphic self-maps, where the relevant operation is composition and the main objects encode the dynamics of the iterates; see for instance \cite{ArosioBracci2016}. In the slice regular setting, however, composition is not the natural structure to preserve regularity, whereas the $*$-product plays a partially analogous role. In the quaternionic case its evaluation formula
\[
(f*g)(x)=f(x)\,g\!\bigl(f(x)^{-1}xf(x)\bigr),
\qquad f(x)\neq 0,
\]
makes this analogy especially suggestive. Nevertheless, the present problem is not dynamical in nature. For this reason, the notion of model considered in this paper is not analytic/dynamical in the sense of iteration theory, but algebraic.

A distinctive feature of slice regularity is indeed the structure of the zero set. If $f$ is not a zero divisor with respect to the $*$-product, then its zeros split into three different types: real zeros, spherical zeros, and isolated nonreal zeros; if instead $N(f)\equiv 0$, then the zero set contains a real surface biholomorphic to $\CC^+$ \cite{Ghiloni2020}. Real and spherical zeros are comparatively rigid and manageable. In fact, when one writes a slice regular function in terms of its slice-preserving components, these zeros correspond to common zeros of those components, hence to commutative holomorphic data on the base. Spherical zeros are encoded by the characteristic polynomials
\[
\Delta_q(x)=(x-q)*(x-q^c)=x^2-x\Tr(q)+N(q),
\]
in close analogy with conjugate pairs of complex zeros \cite{GSSbook}. By contrast, isolated nonreal zeros are much more delicate: the relationship between $*$-factorizations and the actual location of zeros is governed by the camshaft effect \cite{Ghiloni2010}. It is therefore natural to ask whether, inside a suitable equivalence class, one can choose a representative with a simpler arrangement of isolated nonreal zeros.

\medskip
\noindent\textbf{Main results.}
The first notion introduced in the paper is that of a \emph{model}. We say that
a slice regular function \(g\) is a model for \(f\) if \(g\) is strongly
equivalent to \(f\) and all isolated nonreal zeros of \(g\) lie in a single
slice \(\CC_I\). Thus a model is an aligned representative inside the
strong-equivalence class of the original function. This notion is rigid enough
to preserve the invariants \((\Tr,N,\cdiv)\), but flexible enough to capture a
meaningful normalization of the isolated nonreal zero set.

Our first result gives a precise criterion for the existence of such
one-slice-preserving models. More precisely, if \(f=f_0+f_v\) is not
slice-preserving, then \(f\) is strongly equivalent to a \(\CC_I\)-preserving
function if and only if \(N(f_v)\) admits a slice-preserving square root \(h\)
whose divisor coincides with the central divisor of \(f\). In that case
\(f_0+hI\) is a model for \(f\). This criterion, proved in
Theorem~\ref{thm:criterion-one-slice-model}, shows that the existence of models
is governed by the interaction between the square-root problem for \(N(f_v)\)
and the central divisor.

In particular, models do not always exist. Even in the quaternionic polynomial
case, one can construct examples for which no strongly equivalent
one-slice-preserving representative exists; see
Example~\ref{ex:no-CI-preserving-representative}. Thus the strong-equivalence
class is in general too rigid for a complete alignment theorem.

The second main result is positive and concerns purely vectorial functions. We
prove in Theorem~\ref{thm:purely-vectorial-model} that every purely vectorial
slice regular function admits a model. In this case the obstruction described
above disappears: the norm of the vector part can be represented by two
slice-preserving components, and this makes it possible to align all isolated
nonreal zeros in one fixed complex slice.

The bridge from the purely vectorial case to the general case is the
two-direction reduction of Proposition~\ref{prop:two-directions}. It states
that every slice regular function \(f=f_0+f_v\) is strongly equivalent to a
function of the form
\[
f'=f_0+h_1 i+h_2 j,
\]
where \(i,j\) generate a fixed quaternionic subalgebra of \(A\) and \(h_1,h_2\)
are slice-preserving functions. Equivalently, after a strongly equivalent
change of representative, the associated stem function takes values in the
complexification of a fixed three-dimensional subspace
\[
\CC\oplus \CC i\oplus \CC j
\subset A_\CC .
\]
This reduction is the step which allows us to treat the quaternionic and
octonionic cases simultaneously.

For general slice regular functions, the correct alignment statement requires a
weaker notion, which we call a \emph{semi-model}. A semi-model is not required
to be strongly equivalent to the original function, but it preserves the
symmetrization and the real and spherical part of the zero set, while forcing
all isolated nonreal zeros to lie in one slice. In this sense, semi-models
preserve most of the information about zeroes that are relevant to the alignment problem,
while avoiding the central-divisor obstruction.

The main theorem of the paper is the following.

\begin{theorem*}[Existence of semi-models]
Let \(\Omega=\Omega_D^A\) be a basic domain and let \(f\in\Ss(\Omega)\) satisfy
\(N(f)\not\equiv 0\). Then \(f\) admits a semi-model \(\widehat f\). More
precisely,
\[
N(\widehat f)=N(f),
\]
the real and spherical zeros of \(\widehat f\) coincide with those of \(f\), with
the same multiplicities, and all isolated nonreal zeros of \(\widehat f\) are
contained in a single complex slice.
\end{theorem*}

The construction first uses
the two-direction reduction to pass to a representative whose vector part lies
in a fixed quaternionic subalgebra, and then applies the purely vectorial model
theorem after separating the central divisor. In this way every slice regular
function with nontrivial symmetrization admits an aligned representative in a
weaker sense, but 
natural from the algebraic point of view.

The construction also clarifies the different roles of the three invariants.
The trace controls the scalar part; the symmetrization controls the zero
spheres; and the central divisor records the common real and spherical zeros of
the vector part. In particular, the paper makes transparent why the orbit theory
based on \((\Tr,N,\cdiv)\) and the quaternionic weak-equivalence theory based
only on \((\Tr,N)\) have genuinely different content: the former yields a
holomorphic global section of the automorphism orbit bundle, whereas the latter
yields semiregular conjugacy, but in general not a global zero-free regular
conjugator.

The paper is organized as follows. In Section~\ref{sec:equivalence} we recall
the automorphism action on stem functions and the notions of weak and strong
equivalence, emphasizing the difference between global holomorphic sections and
semiregular conjugators. In Section~\ref{sec:models} we introduce models and
semi-models, prove the criterion for one-slice-preserving models, establish the
existence theorem for purely vectorial models, prove the two-direction
reduction, and finally obtain the general semi-model theorem. The last part of
the paper is devoted to explicit quaternionic and octonionic examples
illustrating the gap between models and semi-models and the role of the central
divisor.

\section{Preliminaries}
\subsection{The algebras $\mathbb H$ and $\mathbb O$}
The algebra of octonions is the $8$-dimensional real algebra
\[
\mathbb{O}
=
\left\{
x_0+x_1 e_1+x_2 e_2+x_3 e_3+x_4 e_4+x_5 e_5+x_6 e_6+x_7 e_7 \,|\, x_\ell\in\mathbb{R}
\right\},
\]
where $\{1,e_1,\dots,e_7\}$ is a basis over $\mathbb{R}$, the unit element is $1$, and the imaginary units $e_1,\dots,e_7$ satisfy
\[
e_i^2=-1,
\qquad i=1,\dots,7.
\]
The multiplication is determined by bilinearity together with the rule of the oriented Fano plane in Figure~\ref{fig:fano}: if $(e_i,e_j,e_k)$ is an oriented triple on one of the seven lines, then
\[
e_i e_j = e_k,
\qquad
e_j e_i = -e_k.
\]
In particular, $\mathbb{O}$ is noncommutative and nonassociative, but alternative.

\tikzset{
    midarrow/.style={
        postaction={decorate,
            decoration={
                markings,
                mark=at position #1 with {\arrow{Stealth[scale=1.5]}}
            }
        }
    },
    midarrow/.default=0.5
}
 
\begin{figure}[h]
\centering
\begin{tikzpicture}[
    scale=2,
    every node/.style={circle, draw, fill=white, minimum size=10pt, inner sep=0pt, font=\small}
]
 
    \coordinate (e7) at (0, 1.732);      
    \coordinate (e6) at (-1, 0);         
    \coordinate (e5) at (1, 0);          
    \coordinate (e1) at (-0.5, 0.866);   
    \coordinate (e2) at (0.5, 0.866);    
    \coordinate (e3) at (0, 0);          
    \coordinate (e4) at (0, 0.577);      
 
    \draw[midarrow=0.3, midarrow=0.8] (e6) -- (e1) -- (e7);
 
    \draw[midarrow=0.3, midarrow=0.8] (e7) -- (e2) -- (e5);
 
    \draw[midarrow=0.3, midarrow=0.8] (e5) -- (e3) -- (e6);
 
    \draw[midarrow=0.2, midarrow=0.8] (e3) -- (e4) -- (e7);
 
    \draw[midarrow=0.2, midarrow=0.8] (e1) -- (e4) -- (e5);
 
    \draw[midarrow=0.2, midarrow=0.8] (e2) -- (e4) -- (e6);
 
    \draw[midarrow=0.17, midarrow=0.5, midarrow=0.83]
        (e4) ++(30:0.577) arc[start angle=30, end angle=-330, radius=0.577];
 
    \node at (e7) {$e_7$};
    \node at (e6) {$e_6$};
    \node at (e5) {$e_5$};
    \node at (e1) {$e_1$};
    \node at (e2) {$e_2$};
    \node at (e3) {$e_3$};
    \node at (e4) {$e_4$};
 
\end{tikzpicture}
\caption{The Fano plane with oriented lines encoding the multiplication
of the imaginary units $e_1,\dots,e_7$ of $\OO$. Each directed triple
$e_i \to e_j \to e_k$ on a line gives $e_i e_j = e_k$.}
\label{fig:fano}
\end{figure}
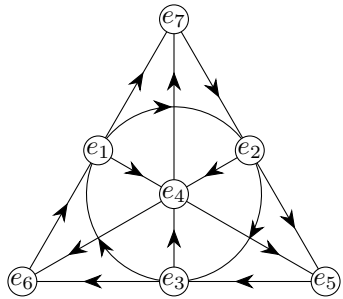
With the above choice, we identify the quaternions $\HH$ with the subalgebra (linearly) generated by $1, e_1,e_2, e_3$.
Accordingly, when convenient, we write
\(e_1=i, e_2=j\),  and \(e_3=k\).

\begin{notation}   
Throughout the paper, $A$ denotes either the quaternions $\HH$ or the octonions $\OO$. 
\end{notation}
We endow $A$ with standard conjugation $$x=x_0+\sum_{\ell=1}^n x_\ell e_\ell\mapsto x^c=x_0-\sum_{\ell=1}^n x_\ell e_\ell,$$
for suitable $n=3,7$.
We set
\[
\Tr(x):=x+x^c,\qquad N(x):=xx^c=x^c x,\qquad
\Re(x):=\frac{\Tr(x)}2=x_0,\qquad \Im(x):=\frac{x-x^c}{2}.
\]
Thus
\[
A=\RR\oplus \Im(A),\qquad x=\Re(x)+\Im(x),\qquad x^c=\Re(x)-\Im(x).
\]
The imaginary part $\Im(x)$ is sometimes denoted as $x_v$ and the writings
$x=x_0+x_v$ is called the \emph{scalar-vector notation}. 
If $x\neq 0$, then $x^{-1}=x^c/N(x)$, and the squared norm is multiplicative:
\[
N(xy)=N(x)N(y).
\]

We set
\[
\mathbb S_A:=\{I\in A:\ I^2=-1\}=\{I\in\Im(A):N(I)=1\},
\qquad
\CC_I:=\Span_\RR\{1,I\}\subset A,
\]
and we also write \(\CC_I^+:=\{x+yI:\ x\in\RR,\ y>0\}\).
Therefore, $\SF_\HH$ is a $2$-sphere, while $\SF_\OO$ is a $6$-sphere, while the algebra $A$ can be viewed as an $\SF_A$-union of complex planes:
$$A=\bigcup_{I\in\SF_A}\CC_I$$.

On $\Im(A)$ we use the Euclidean inner and cross products, namely
\[
\langle u,v\rangle=-\Re(uv),\qquad u\times v:=\frac12(uv-vu),\qquad u,v\in\Im(A),\]
and therefore, if $x=x_0+x_v,y=y_0+y_v\in A$, then
$$xy=x_0y_0-\langle x_v,y_v\rangle+x_0y_v+y_0x_v+x_v \times y_v.$$

Let $\ii\in\CC$ be the standard imaginary unit. The complexification of $A$ is
\[
A_\CC:=A\otimes_\RR\CC.
\]
Each $z\in A_\CC$ can be written uniquely as $z=a+\ii b$, with $a,b\in A$.
We denote by $z^c:=a^c+\ii b^c$ the $\CC$-linear extension of the conjugation of $A$, and by
\[
\overline z:=a-\ii b
\]
the complex conjugation. We keep the notation
\[
\Tr(z):=z+z^c,\qquad N(z):=zz^c.
\]
Recall that $z\in A_\CC$ is invertible if and only if $N(z)\neq 0$, in which case \(z^{-1}=\frac{z^c}{N(z)}\).
\begin{remark}
Although $A$ is a division algebra, its complexification $A_\CC$ is not. Indeed, every nonzero element
$z\in A_\CC$ such that $N(z)=0$ is a zero divisor. In particular, for any $I\in\mathbb S_A$, the element
\(E:=\frac{1+\ii I}{2}\)
is an idempotent:
\[
E^2=E,
\qquad
N(E)=EE^c=0.
\]
\end{remark}


\subsection{Slice Regularity}
We will define and discuss slice regularity using the recent approach developed in~\cite{AltavillaLAA,AM:powercover} based on that of \textit{stem functions} to~\cite{ghiloni2011slice}.
For the classical direct approach we refer to~\cite{GSS:RegFunc}.

\begin{definition}
Let $\Dd\subset\CC$ be an \emph{open set} such that $\overline{\Dd}=\Dd$.
The \emph{circularization} of $\Dd$ in $A$ is the set
\[
\Omega_\Dd^A
:=\{\alpha+\beta I\in A\,|\,\alpha+\ii\beta\in\Dd,\, I\in\mathbb S_A\}.
\]
\end{definition}
\begin{remark}\label{rem:base-projection}
Whenever $\Omega=\Omega_\Dd^A$, we write $D_\Omega:=\Dd$ and we denote by
\[
\pi:\Omega\longrightarrow \Dd,\qquad \pi(\alpha+\beta I)=\alpha+\ii\beta,
\]
the natural projection onto the base.
\end{remark}

\begin{definition}
Let $\Omega\subset A$ be a domain. We say that $\Omega$ is \emph{circular} (or \emph{axially symmetric})
if for every $q=\alpha+\beta I\in\Omega$ (with $\alpha,\beta\in\RR$, $I\in\mathbb S_A$) the whole sphere
\[
\mathbb S_q:=\{\alpha+\beta J\,|\, J\in\mathbb S_A\}
\]
is contained in $\Omega$.
\end{definition}

\begin{definition}\label{def:basic-domain}
Let $\Omega\subseteq A$ be a circular domain and assume $\Omega=\Omega_\Dd^A$.
We say that $\Omega$ is a \emph{basic domain} if, for every $J\in\mathbb S_A$, the slice
\(\Omega_J^+:=\Omega\cap\CC_J^+\)
is simply connected (equivalently, $\Dd=D_\Omega$ is simply connected).
\end{definition}

We are ready to define stem functions and slice regular functions.
\begin{definition}
Let $\Dd\subseteq\CC$ be an open domain such that $\bar \Dd=\Dd$.
A \emph{stem function} is a map
$F:\Dd\to A_{\CC}$ of the form $F=F_1+\ii F_2$ with $F_1,F_2:\Dd\to A$, satisfying the intrinsic condition
\[
F(\overline{z})=\overline{F(z)}\qquad(z\in\Dd),
\]
equivalently $F_1(\overline{z})=F_1(z)$ and $F_2(\overline{z})=-F_2(z)$.
The stem function $F$ induces a \emph{slice function} $f=\Ii(F):\Omega_\Dd^A\to A$ defined by
\[
f(\alpha+I\beta)=F_1(\alpha+\ii\beta)+I\,F_2(\alpha+\ii\beta),
\qquad \alpha+\ii\beta\in\Dd,\ I\in\mathbb S_A.
\]
A slice function $f=\Ii(F)$ is called \emph{slice regular} if $F$ is holomorphic on $\Dd$.
For a circular domain $\Omega=\Omega_\Dd^A$ we denote by $\Ss(\Omega)$ the set of slice regular functions.
\end{definition}

Given a slice regular function $f=\Ii(F)$, we set $f^c=\Ii(F^c)$.

An important subset of $\Ss(\Omega)$ is that of slice preserving functions.

\begin{definition}
    Let $\Dd\subseteq\CC$ be an open domain such that $\bar \Dd=\Dd$. A slice regular function $f=\mathcal I(F):\Omega_\Dd^A\to A$ is said to be slice preserving if one of the following equivalent conditions is satisfied:
    \begin{itemize}
        \item for any $I\in\SF_A$ it holds that $f(\Omega_\Dd^A\cap\CC_I)\subset\CC_I$;
        \item the stem function $F=F_1+\sqrt{-1}F_2$ takes value in $\CC$ (i.e., $F_1,F_2\in\RR$);
        \item for any $x\in\Omega_\Dd^A$, it holds that $f(x)^c=f(x^c)$.
    \end{itemize}
The set of slice-preserving regular functions on $\Omega$ will be denoted by $\Ss_\RR(\Omega)$.
\end{definition}

\subsection{The $*$-product and the complex-curve viewpoint}

Let $\Omega=\Omega_D^A$ be a circular domain.
The pointwise sum turns $\Ss(\Omega)$ into a real vector space, and if
$f=\Ii(F)$ and $g=\Ii(G)$ are slice regular, their \emph{$*$-product} is defined by
\[
f*g:=\Ii(FG),
\]
where $FG$ denotes the pointwise product in $A_\CC$.
If $A=\HH$, then the $*$-product is associative; if $A=\OO$, it is  nonassociative, but alternative.
Moreover, $\Ss_\RR(\Omega)$ is  the center of the algebra $(\Ss(\Omega),+,*)$.

\begin{remark}
Fix an orthonormal basis of $A$ of the form
\[
\{1,e_1,\dots,e_n\},
\]
where $n=3$ if $A=\HH$ and $n=7$ if $A=\OO$.
Then every stem function $F:D\to A_\CC$ can be written uniquely as
\[
F=F_0+F_1e_1+\dots+F_ne_n,
\]
with $F_0,\dots,F_n:D\to\CC$ intrinsic holomorphic functions.
Accordingly, one associates with $F$ the complex curve
\[
\gamma_F:D\to\CC^{n+1},\qquad \gamma_F=(F_0,F_1,\dots,F_n).
\]
It is immediate that $F$ is holomorphic if and only if $\gamma_F$ is holomorphic.

Since every complex-valued stem function induces a slice-preserving function, it follows from
\cite[Proposition 3.12]{Colombo2013} and \cite[Lemma 6.11]{GMP} that
$\Ss(\Omega)$ is a free $\Ss_\RR(\Omega)$-module of rank $\dim_\RR A$.

More concretely, if $f=\Ii(F)$ and
\[
F=F_0+F_1e_1+\dots+F_ne_n,
\]
then
\[
f=f_0+f_1e_1+\dots+f_ne_n,
\qquad\text{where}\qquad
f_\ell:=\Ii(F_\ell)\in\Ss_\RR(\Omega).
\]
This point of view has been extensively exploited in
\cite{AltavillaPAMS,AltavillaAMPA,AltavillaLAA,AltavillaLINCEI,AM:powercover,AltavillaMongodi},
since it provides a very effective way to represent and study the $*$-product.
\end{remark}

\begin{remark}
In the sequel, we will often tacitly identify a slice regular function $f=\Ii(F)$ with the associated
complex curve $\gamma_F$, or equivalently with the tuple of its slice-preserving components.
Thus we shall write
\[
f=f_0+f_1e_1+\dots+f_ne_n,
\]
and we set
\[
f_v:=f_1e_1+\dots+f_ne_n
\]
for the vector part of $f$.
\end{remark}

With this notation, if $f=f_0+f_v$ and $g=g_0+g_v$ are slice regular functions, then
\begin{align*}
f+g&=(f_0+g_0)+(f_v+g_v),\\
f*g&=f_0g_0-\langle f_v,g_v\rangle+f_0g_v+g_0f_v+f_v\times g_v,
\end{align*}
where $\langle\cdot,\cdot\rangle$ and $\times$ are computed pointwise on the vector parts.

\begin{remark}
Given a slice regular function $f=f_0+f_v$, one has
\(f^c=f_0-f_v\).
Hence
\[
f_0=\frac{f+f^c}{2},
\qquad
f_v=\frac{f-f^c}{2},
\qquad
\Tr(f)=f+f^c=2f_0.
\]
Moreover, $f$ is slice preserving if and only if $f=f^c$.
\end{remark}

\begin{example}
Polynomials and convergent power series of the form
\[
f(x)=\sum_{k=0}^N x^k a_k,
\qquad N\in\NN\cup\{\infty\},\quad a_k\in A,
\]
are slice regular on their natural domains of definition.

Writing
\(a_k=a_{k,0}+\sum_{\ell=1}^n a_{k,\ell}e_\ell\),
where \(a_{k,\ell}\in\RR\),
we obtain
\[
f(x)
=
\sum_{k=0}^N x^k a_{k,0}
+
\sum_{\ell=1}^n\left(\sum_{k=0}^N x^k a_{k,\ell}\right)e_\ell,
\]
so that each component function
\[
f_\ell(x):=\sum_{k=0}^N x^k a_{k,\ell}
\]
belongs to $\Ss_\RR(\Omega)$.

Moreover,
\[
f^c(x)=\sum_{k=0}^N x^k a_k^c
=
\sum_{k=0}^N x^k a_{k,0}
-
\sum_{\ell=1}^n\left(\sum_{k=0}^N x^k a_{k,\ell}\right)e_\ell.
\]
\end{example}

\begin{definition}
Let $f=f_0+f_v\in\Ss(\Omega)$.
If there exist $I\in\mathbb S_A$ and $h\in\Ss_\RR(\Omega)$ such that
\[
f_v=h\,I,
\]
then we say that $f$ is \(\CC_I\)\emph{-preserving}.
The subalgebra of $\CC_I$-preserving functions will be denoted by $\Ss_I(\Omega)$.
\end{definition}

Examples of $\CC_I$-preserving functions are convergent power series whose coefficients all lie in $\CC_I$.

\begin{remark}
If $f,g\in \Ss_I(\Omega)$, then $f*g=g*f$.
More generally, if the vector parts of $f$ and $g$ are linearly dependent over $\Ss_\RR(\Omega)$, then
$f_v\times g_v\equiv 0$, and therefore
\(f*g=g*f\).
\end{remark}

\begin{definition}
Given $f\in\Ss(\Omega)$, we define its \emph{symmetrization} (or \emph{symmetrized function}) as
\[
N(f):=f*f^c\in \Ss_\RR(\Omega).
\]
\end{definition}

Explicitly, if
\(f=f_0+\sum_{\ell=1}^n f_\ell e_\ell
\),
then
\[
N(f)=f_0^2+\sum_{\ell=1}^n f_\ell^2.
\]
In particular, $N(f)$ is slice preserving.
Sometimes, $N(f)$ is also denoted by $f^s$.

There exist slice regular functions $f\not\equiv 0$ such that $N(f)\equiv 0$ (see e.g.~\cite{AltavillaLAA}).
This phenomenon can occur only when $\Omega\cap\RR=\emptyset$.
Indeed, if $r\in\Omega\cap\RR$, then each component $f_\ell(r)$ is real, so
\[
N(f)(r)=f_0(r)^2+\sum_{\ell=1}^n f_\ell(r)^2
\]
is a sum of squares of real numbers.
Hence $N(f)(r)=0$ if and only if $f_\ell(r)=0$ for every $\ell=0,\dots,n$.
If this happens on a set with an accumulation point, then the identity principle implies
$f_\ell\equiv 0$ for all $\ell$, and therefore $f\equiv 0$.



\subsection{Zeros of slice regular functions}\label{subsec:zeros}

We recall only the basic facts about the zero set of slice regular functions needed in the sequel (see, e.g.,
\cite[Chapter~3]{GSSbook} and \cite{ghiloni2011slice}).
Given $f\in\Ss(\Omega)$, we denote by $\Zz(f)$ its zero set.

For $q=\alpha+\beta I\in A$, the corresponding \emph{characteristic polynomial} is
\[
\Delta_q(x):=x^2-\Tr(q)\,x+N(q).
\]
It depends only on the sphere
\(\mathbb S_q:=\{\alpha+\beta J:\ J\in\mathbb S_A\}\)
and satisfies
\(\Zz(\Delta_q)=\mathbb S_q\).

Let $f\in\Ss(\Omega)$ be defined on a circular domain.
If $f(q)=0$, then the intersection $\Zz(f)\cap \mathbb S_q$ is either the singleton $\{q\}$ or the whole sphere $\mathbb S_q$.
Accordingly, nonreal zeros split into two types: \emph{isolated zeros} and \emph{spherical zeros}.

Moreover, if $N(f)\equiv 0$ and $f\neq 0$, then $\Zz(f)$ contains a real surface biholomorphic to upper half complex plane.
Hence, throughout the sequel, we will always assume
\[
N(f)\not\equiv 0.
\]

A basic factorization result is the following.

\begin{theorem}[\cite{GSSbook,Ghiloni2010}]\label{thm:decomposition}
Let $f\in\Ss(\Omega)$ be such that $N(f)\not\equiv 0$, and let $q\in\Omega\setminus\RR$ be such that
$\mathbb S_q\subset\Omega$.
Then there exist integers $m,n\in\NN\cup\{0\}$ and points $p_1,\dots,p_n\in\mathbb S_q$, with
$p_i\neq p_{i+1}^c$ for every $i=1,\dots,n-1$, such that
\begin{equation}\label{eq:decomp}
f(x)=\Delta_q(x)^m\,(x-p_1)*\cdots*(x-p_n)*g(x),
\end{equation}
where $g\in\Ss(\Omega)$ satisfies \(\Zz(g)\cap\mathbb S_q=\emptyset\).
\end{theorem}

\begin{definition}\label{def:multiplicities}
Under the assumptions of Theorem~\ref{thm:decomposition}, we make the following definitions.

\begin{itemize}
\item If $\mathbb S_q\subset\Zz(f)$, then $2m$ is called the \emph{spherical multiplicity} of the spherical zero $\mathbb S_q$.

\item If $\Zz(f)\cap\mathbb S_q=\{q\}$, then $n$ is called the \emph{isolated multiplicity} of the isolated zero $q$.
\end{itemize}

If $q_0\in\Omega\cap\RR$, the \emph{isolated multiplicity of $f$ at $q_0$} is the integer $k\in\NN\cup\{0\}$ such that
\[
f(x)=(x-q_0)^k\,g(x),
\]
for some $g\in\Ss(\Omega)$ satisfying $g(q_0)\neq 0$.
\end{definition}

\begin{remark}
If \(f\in \Ss(\Omega)\) is slice preserving, then its zero set \(\Zz(f)\) consists only of real zeros and spherical zeros. Instead, if \(f\) is \(\CC_I\)-preserving for some \(I\in \SF_A\), then every isolated nonreal zero of \(f\) belongs to the slice \(\CC_I\).
\end{remark}

\begin{remark}\label{rem_sph_real_zeroes}
Let
\[
f=f_0+f_1e_1+\dots+f_ne_n
\]
be a slice regular function on a basic domain.
Since $\Delta_{q_0}$ is slice preserving, the sphere $\mathbb S_{q_0}$ is a spherical zero of $f$ if and only if it is a common zero set
of all the components $f_\ell$, for $\ell=0,\dots,n$.
The same argument applies to real isolated zeros.

Therefore, zeros of these two types can be extracted from $f$ by considering a Weierstra\ss\ factorization of the stem\footnote{Since each $f_\ell$ is slice preserving, the inducing stem function $F_\ell$ is a genuine complex holomorphic function defined on a simply connected domain, or on a union of two symmetric simply connected domains.}
$F_\ell$ of each component $f_\ell$, and then taking the common polynomial factor.
\end{remark}

\subsection{The central divisor}

In view of Remark~\ref{rem_sph_real_zeroes}, we now recall and reinterpret the notion of
\emph{central divisor} introduced in~\cite{Bisi2025,Bisioctonions}.

Let
\[
\varphi=(\varphi_1,\dots,\varphi_n):D\to\CC^n
\]
be a meromorphic curve, that is, each component $\varphi_j$ is meromorphic on $D$.
For every point $p\in D$, we define
\[
\ord_p(\varphi):=\min_{1\le j\le n}\ord_p(\varphi_j),
\]
where $\ord_p(\varphi_j)\in\ZZ$ is the usual order of $\varphi_j$ at $p$
(positive for zeros and negative for poles).

The \emph{divisor} of $\varphi$ is then
\[
\operatorname{div}(\varphi):=\sum_{p\in D}\ord_p(\varphi)\,[p].
\]
Its positive part records the common zeros of the components of $\varphi$,
while its negative part records their common poles.

\begin{definition}
Let $f=f_0+f_v\in\Ss(\Omega)$, and let
\[
f_v=\Ii(F_v),
\qquad
F_v:D_\Omega\to \Im(A)\otimes_\RR\CC \simeq \CC^n
\]
be the stem function inducing the vector part of $f$.
The \emph{central divisor} of $f$ is the divisor
\[
\cdiv(f):=\operatorname{div}(F_v).
\]
Since $F_v$ is holomorphic, $\cdiv(f)$ is an effective divisor.
\end{definition}

\begin{remark}
Writing
\[
F_v=F_1e_1+\dots+F_ne_n,
\]
the divisor $\cdiv(f)$ is exactly the common zero divisor of the intrinsic holomorphic functions
$F_1,\dots,F_n$.
Equivalently, if
\[
f_v=f_1e_1+\dots+f_ne_n
\]
is the decomposition of the vector part into slice-preserving components, then $\cdiv(f)$ records the common zeros
of the stem functions inducing $f_1,\dots,f_n$.

Since the functions $F_1,\dots,F_n$ are intrinsic holomorphic, their common zeros are either real points of $D_\Omega$
or conjugate pairs $z,\bar z$.
Under circularization, these correspond exactly to real zeros or spherical zeros of $f_v$.
Therefore, the central divisor of $f$ can be interpreted as the divisor encoding the real and spherical zeros
of the vector part $f_v$, counted with multiplicities.
\end{remark}

\begin{remark}
Notice that $f_v(x)$ may vanish at some point $x\in\Omega$ even though the corresponding base point
$\pi(x)$ does not belong to the support of $\cdiv(f)$.
Indeed, $\cdiv(f)$ only records the common zeros of the components of the curve
\[
(F_1,\dots,F_n),
\]
whereas additional zeros of $f_v$ may appear because of the algebraic structure of $A_\CC$.
In other words, the vanishing of the vector part as an element of $A_\CC$ is in general weaker than the simultaneous
vanishing of all its complex components.
Such a $x$ correspond exactly to a non real isolated zero for $f$.
\end{remark}

\begin{example}
Consider the slice regular polynomials on $\HH$
\[
P(x)=(x-2i)j,
\qquad
Q(x)=\Delta_i(x)\,i.
\]
Then:
\begin{itemize}
\item $\Zz(P)=\{2i\}$, while $\cdiv(P)=0$;
\item $\Zz(Q)=\mathbb S_i$, while $\cdiv(Q)=[\ii]+[-\ii]$ on the base $\CC$;
\item $\Zz(\Delta_i P)=\mathbb S_i\cup\{2i\}$, while
\[
\cdiv(\Delta_i P)=[\ii]+[-\ii].
\]
\end{itemize}
Thus the central divisor detects the spherical factor $\Delta_i$, but it does not detect the isolated zero $2i$.
\end{example}

\subsection{Regular inverses and semiregular functions}

We conclude the preliminaries by recalling the notions needed later for
$*$-quotients and conjugation by possibly nonregular factors
(see, e.g., \cite{GSSbook,Ghiloni2017}).

\begin{definition}\label{def:regular-inverse}
Let $f\in\Ss(\Omega)$ be such that $N(f)$ does not vanish on $\Omega$.
The \emph{regular inverse} of $f$ is the slice regular function
\[
f^{-*}:=(N(f))^{-1}*f^c .
\]
It satisfies
\(f^{-*}*f=f*f^{-*}=1\).
\end{definition}

\begin{remark}
In particular, if $f\in\Ss(\Omega)$ is zero-free, then $f^{-*}$ is again slice regular on $\Omega$.
More generally, if $f$ has isolated zeros, then $f^{-*}$ is defined and slice regular away from those zeros,
and it develops poles at them.
\end{remark}

\begin{definition}\label{def:semiregular}
Let $\Omega\subseteq A$ be a circular domain.
A function $f:\Omega\to A$ is called \emph{semi-regular} if there exists a discrete set
$P\subset\Omega$ such that
\[
f\in \Ss(\Omega\setminus P),
\]
and every point of $P$ is a pole of $f$.
The set of semi-regular functions on $\Omega$ will be denoted by $\Mm(\Omega)$.
\end{definition}

\begin{remark}
Equivalently, semi-regular functions are those functions which are slice regular outside a discrete set
and have no essential singularities.
In particular, if $g,h\in\Ss(\Omega)$ and $h\not\equiv 0$, then the $*$-quotient
\[
g*h^{-*}
\]
is semi-regular on $\Omega$, with poles contained in the zero set of $N(h)$.
\end{remark}

\begin{remark}
Later on, semi-regular functions will appear mainly as conjugators in identities of the form
\[
f=\chi^{-*}*g*\chi,
\qquad \chi\in\Mm(\Omega).
\]
For this reason, the above minimal definition is sufficient for our purposes.
\end{remark}

\section{Automorphisms and equivalence}\label{sec:equivalence}

In this section we describe the natural action of holomorphic families of automorphisms of the
complexified algebra $A_\CC$ on stem functions, and we explain how this action is governed by the
three invariants
\((\Tr, N, \cdiv)\).
This provides the conceptual framework for the notions of weak and strong equivalence and will be the
starting point for the construction of models and semi-models in the next section.

\subsection{The automorphism action on stem functions}

We first clarify the terminology used in this subsection. The word
\emph{intrinsic} always refers to compatibility with the real structures on the
complex base and on the complexified algebra. Thus a holomorphic map
\[
F:D\longrightarrow A_\CC
\]
is intrinsic if
\[
F(\bar z)=\overline{F(z)}
\qquad\text{for every }z\in D.
\]
In the language of slice
regular functions, these intrinsic holomorphic maps are precisely the stem
functions inducing slice functions on the corresponding circular domain.

We denote by
\[
G_\CC:=\Aut(A_\CC)
\]
the group of $\CC$-linear algebra automorphisms of the complexified algebra $A_\CC$.

If
\(F:D\longrightarrow A_\CC\)
is a stem function and
\(\phi:D\longrightarrow G_\CC\)
is holomorphic, we define their pointwise action by
\[
(\phi\cdot F)(z):=\phi(z)\bigl(F(z)\bigr), \qquad z\in D.
\]
In general, even if $F$ is intrinsic, the map $\phi\cdot F$ need not be intrinsic.
Therefore one has to impose a compatibility condition between $\phi$ and the real structure of the base and of the fiber.

\begin{definition}\label{def:intrinsic-automorphism-section}
A holomorphic map
\(\phi:D\longrightarrow G_\CC\)
is called \emph{intrinsic} if
\[
\phi(\bar z)(\bar w)=\overline{\phi(z)(w)}
\qquad \forall z\in D,\ \forall w\in A_\CC .
\]
\end{definition}

\begin{remark}
Equivalently, if one endows the trivial holomorphic bundle
\(D\times G_\CC \longrightarrow D\)
with the real structure
\[
\sigma(z,\Phi):=(\bar z,\bar\Phi),
\qquad
\bar\Phi(w):=\overline{\Phi(\bar w)},
\]
then an intrinsic holomorphic map $\phi:D\to G_\CC$ is precisely a holomorphic section fixed by $\sigma$.
\end{remark}
\begin{remark}
The two uses of the word intrinsic are compatible. If \(F\) is an intrinsic
stem function and \(\phi:D\to G_\CC\) is an intrinsic automorphism-valued map,
then \(\phi\cdot F\) is again intrinsic, hence it induces a slice function.
\end{remark}

\begin{proposition}\label{prop:intrinsic-action}
Let $F:D\to A_\CC$ be an intrinsic holomorphic stem function and let
\(\phi:D\to G_\CC\)
be an intrinsic holomorphic map. Then $\phi\cdot F$ is again an intrinsic holomorphic stem function.
Consequently, if $f=\Ii(F)\in\Ss(\Omega_D^A)$, then
\[
\phi\cdot f:=\Ii(\phi\cdot F)
\]
is a well-defined slice regular function on $\Omega_D^A$.
\end{proposition}

\begin{proof}
Since both $F$ and $\phi$ are holomorphic and the action
\[
G_\CC\times A_\CC\longrightarrow A_\CC,\qquad (\Phi,w)\mapsto \Phi(w),
\]
is holomorphic, the map $\phi\cdot F$ is holomorphic on $D$.

It remains to prove the intrinsic condition. For every $z\in D$ we have
\[
(\phi\cdot F)(\bar z)
=
\phi(\bar z)\bigl(F(\bar z)\bigr).
\]
Since $F$ is intrinsic, then
\(F(\bar z)=\overline{F(z)}\).
Since $\phi$ is intrinsic, applying Definition~\ref{def:intrinsic-automorphism-section} with $w=F(z)$ gives
\[
\phi(\bar z)\bigl(\overline{F(z)}\bigr)
=
\overline{\phi(z)\bigl(F(z)\bigr)}.
\]
Therefore
\((\phi\cdot F)(\bar z)
=
\overline{(\phi\cdot F)(z)}\),
which proves that $\phi\cdot F$ is an intrinsic stem function.
Hence $\Ii(\phi\cdot F)$ is slice regular on $\Omega_D^A$.
\end{proof}

\begin{remark}
Proposition~\ref{prop:intrinsic-action} shows that intrinsic holomorphic maps
\(\phi:D\to \Aut(A_\CC)\)
act naturally on slice regular functions through their stem functions.
This is the action that will appear in the characterization of strong equivalence below.
\end{remark}
\subsection{Weak and strong equivalence, and the orbit bundle}

We now introduce the two equivalence relations that will be used throughout the paper and
state the geometric meaning of strong equivalence in terms of holomorphic sections of a natural
bundle over the base.

\begin{definition}\label{def:weak-strong}
Let $f,g\in \Ss(\Omega)$ be slice regular functions.
We say that $f$ and $g$ are \emph{weakly equivalent} if
\[
\Tr(f)\equiv \Tr(g)
\qquad\text{and}\qquad
N(f)\equiv N(g),
\]
and we say that $f$ and $g$ are \emph{strongly equivalent} if, in addition,
\[
\cdiv(f)=\cdiv(g).
\]
\end{definition}

\begin{remark}
If $f$ and $g$ are slice preserving, then strong equivalence reduces to equality.
Indeed, in that case the vector part vanishes identically, so $\cdiv$ plays no role, and
the pair $(\Tr,N)$ already determines the function.
\end{remark}

\begin{theorem}\label{thm:strong-equivalence-sections}
Let $\Omega=\Omega_D^A$ be a basic domain, and let
\(f=\Ii(F), h=\Ii(H)
\)
be slice regular functions with associated stem functions
\(
F,H:D\longrightarrow A_\CC\).

Assume first that neither $f$ nor $g$ is slice preserving.
Then the following are equivalent:
\begin{enumerate}
\item $f$ and $g$ are strongly equivalent;
\item $F$ and $H$ have the same invariants $\Tr$, $N$, and $\cdiv$;
\item for every $z\in D$ there exists an element $\alpha_z\in \Aut(A_\CC)$ such that
\(F(z)=\alpha_z\bigl(H(z)\bigr)\);
\item there exists a holomorphic map
\(\phi:D\longrightarrow \Aut(A_\CC)
\)
such that
\[
F(z)=\phi(z)\bigl(H(z)\bigr)
\qquad \forall z\in D.
\]
\end{enumerate}

If, on the other hand, $f$ is slice preserving, then the previous conditions are equivalent to
\[
f=h.
\]
\end{theorem}

\begin{proof}
For $A=\HH$ this is exactly Theorem~1.1 in~\cite{Bisi2025}, while for $A=\OO$
it is Theorem~2.1 in~\cite{Bisioctonions}.
\end{proof}

\begin{remark}\label{rem:orbit-bundle}
Theorem~\ref{thm:strong-equivalence-sections} admits a very natural bundle-theoretic reformulation.
Define
\[
V_{f,h}
:=
\bigl\{(z,\alpha)\in D\times G_\CC \ :\ F(z)=\alpha\bigl(H(z)\bigr)\bigr\}.
\]
Let
\[
\pi_{f,h}:V_{f,h}\longrightarrow D,
\qquad
\pi_{f,h}(z,\alpha)=z,
\]
be the natural projection.

Then condition \emph{(3)} in Theorem~\ref{thm:strong-equivalence-sections}
simply says that every fiber of $\pi_{f,h}$ is nonempty, whereas condition \emph{(4)}
says that $\pi_{f,h}$ admits a holomorphic global section.
Indeed, a holomorphic map
\(\phi:D\longrightarrow G_\CC\)
satisfying
\[
F(z)=\phi(z)\bigl(H(z)\bigr)
\qquad \forall z\in D
\]
is the same thing as a holomorphic section
\[
\sigma_\phi:D\longrightarrow V_{f,h},
\qquad
\sigma_\phi(z):=(z,\phi(z)).
\]

Therefore strong equivalence can be read as the existence of a holomorphic global section of the
orbit bundle $\pi_{f,h}$.
\end{remark}

\begin{remark}\label{rem:why-bundle-viewpoint}
The bundle language is not merely cosmetic: it is exactly the language used in the proofs of the
results quoted above.
In particular, in the octonionic case one studies the projection
\[
\pi:V\to D,
\qquad
V=\{(\alpha,z)\in G_\CC\times D:\ F(z)=\alpha(H(z))\},
\]
first proving the existence of local holomorphic sections and then upgrading them to a global
holomorphic section by a combination of topological arguments and Oka theory.
Thus, for our purposes, the most useful interpretation of strong equivalence is that it provides
a holomorphic section of a natural orbit bundle over the base $D$.
\end{remark}

\subsection{The quaternionic case: inner conjugation and weak equivalence}

In the quaternionic case the orbit picture described in the previous subsection
becomes much more concrete.
Indeed, every $\CC$-algebra automorphism of $\HH_\CC$ is inner, so that a holomorphic
section of the orbit bundle may be lifted locally to a holomorphic map with values in
$\HH_\CC^\ast$.
Passing from stem functions to slice regular functions, this produces local zero-free
slice regular conjugators.
On the other hand, if one keeps only the invariants $\Tr$ and $N$ and forgets the
central divisor, then conjugation still exists, but in general only at the semiregular
level.

\begin{theorem}\label{thm:weak-equivalence-quaternionic}
Assume $A=\HH$ and let $f,g\in \Ss(\Omega)$.
Then the following are equivalent:
\begin{enumerate}
\item $f$ and $g$ are weakly equivalent;
\item there exists a semiregular function $\chi\in \Mm(\Omega)$, not a zero divisor, such that
\[
f=\chi^{-*}*g*\chi .
\]
\end{enumerate}
\end{theorem}

\begin{proof}
Write
\(f=f_0+f_v, g=g_0+g_v
\).
Since
\(\Tr(f)=2f_0, \Tr(g)=2g_0\),
and
\(N(f)=f_0^2+f_v^{\,s},N(g)=g_0^2+g_v^{\,s}\),
the condition that $f$ and $g$ are weakly equivalent is equivalent to
\(f_0=g_0, f_v^{\,s}=g_v^{\,s}\).
For quaternionic semi-regular functions, this is exactly the characterization of
equivalence under $*$-conjugation proved in \cite[Corollary~7.2]{AltavillaLAA}.
Since slice regular functions are, in particular, semiregular, the conclusion follows.
\end{proof}

\begin{proposition}\label{prop:strong-local-conjugators}
Assume $A=\HH$ and let
\(f=\mathcal I(F), h=\mathcal I(H)\in \mathcal S(\Omega)\), with \(\Omega=\Omega_D^\HH
\),
be strongly equivalent on a basic domain.
Choose a holomorphic section \(
\phi:D\to\Aut(\HH_\C)\)
such that
\[
F(z)=\phi(z)(H(z))
\qquad \forall z\in D.
\]
Then there exists a symmetric open cover $\{U_\lambda\}$ of $D$ and local intrinsic
holomorphic liftings
\[
\alpha_\lambda:U_\lambda\to\HH_\C^\ast
\]
of $\phi$, namely
\[
\phi(z)(w)=\alpha_\lambda(z)^{-1}w\alpha_\lambda(z),
\]
such that the induced slice regular functions
\[
\chi_\lambda:=\mathcal I(\alpha_\lambda)\in\mathcal S(\Omega_{U_\lambda}^\HH)
\]
are zero-free and satisfy
\[
f=\chi_\lambda^{-*}*g*\chi_\lambda
\qquad\text{on }\Omega_{U_\lambda}^\HH.
\]
Moreover, on each overlap $U_{\lambda\mu}=U_\lambda\cap U_\mu$ there exists
$\mu_{\lambda\mu}\in\mathcal S_\RR(\Omega_{U_{\lambda\mu}}^\HH)^\ast$ such that
\[
\chi_\lambda=\mu_{\lambda\mu}\chi_\mu.
\]

\end{proposition}
Before performing the proof we clarify a detail of the statement.
\begin{remark}\label{rem:transition-centralizer}
The slice-preserving nature of the transition functions in
Proposition~\ref{prop:strong-local-conjugators} is a consequence of the fact that
the local conjugators are obtained from local liftings of one and the same
holomorphic section $\phi:D\to\Aut(\HH_\C)$.

If instead $\chi_\lambda$ and $\chi_\mu$ are two arbitrary local zero-free
conjugators satisfying
\[
f=\chi_\lambda^{-*}*g*\chi_\lambda,
\qquad
f=\chi_\mu^{-*}*g*\chi_\mu,
\]
then the ratios
\[
\chi_\mu*\chi_\lambda^{-*}
\qquad\text{and}\qquad
\chi_\lambda^{-*}*\chi_\mu
\]
need not be slice preserving. One only has
\[
(\chi_\mu*\chi_\lambda^{-*})*g
=
g*(\chi_\mu*\chi_\lambda^{-*}),
\]
and
\[
f*(\chi_\lambda^{-*}*\chi_\mu)
=
(\chi_\lambda^{-*}*\chi_\mu)*f.
\]
So arbitrary ratios lie in the $*$-centralizer of $g$ or $f$, respectively.
\end{remark}
\begin{proof}[proof of Proposition~\ref{prop:strong-local-conjugators}]
By Theorem~\ref{thm:strong-equivalence-sections}, strong equivalence yields a holomorphic
section of the orbit bundle
\[
\pi_{f,h}:V_{f,h}\to D,
\qquad
V_{f,h}=\{(z,\phi)\in D\times \Aut(\HH_\CC):\ F(z)=\phi(z)(H(z))\}.
\]
Equivalently, there exists a holomorphic map
\(\phi:D\to \Aut(\HH_\CC)
\)
such that
\[
F(z)=\phi(z)(H(z))
\qquad \forall z\in D.
\]

Since every automorphism of $\HH_\CC$ is inner, for every point of $D$ there exists
a symmetric open neighborhood $U_\lambda\subset D$ and a holomorphic map
\[
\alpha_\lambda:U_\lambda\to \HH_\CC^\ast
\]
such that
\[
\phi(z)(w)=\alpha_\lambda(z)^{-1}\,w\,\alpha_\lambda(z)
\qquad\text{for all }z\in U_\lambda,\ w\in \HH_\CC .
\]
After choosing these liftings compatibly with the real structure, they induce zero-free
slice regular functions
\[
\chi_\lambda:=\Ii(\alpha_\lambda)\in \Ss(\Omega_{U_\lambda}^{\HH}),
\]
and the identity
\[
F(z)=\alpha_\lambda(z)^{-1}H(z)\alpha_\lambda(z)
\]
translates into
\[
f=\chi_\lambda^{-*}*h*\chi_\lambda
\qquad\text{on }\Omega_{U_\lambda}^{\HH}.
\]

Now let $z\in U_{\lambda\mu}$.
Since $\alpha_\lambda(z)$ and $\alpha_\mu(z)$ induce the same inner automorphism of
$\HH_\CC$, their ratio belongs to the center of $\HH_\CC^\ast$, namely to $\CC^\ast$.
Hence there exists a holomorphic map
\[
\lambda_{\lambda\mu}:U_{\lambda\mu}\to \CC^\ast
\]
such that
\[
\alpha_\lambda=\lambda_{\lambda\mu}\,\alpha_\mu .
\]
Because the liftings are compatible with the real structure, $\lambda_{\lambda\mu}$ is
intrinsic, hence
\[
\mu_{\lambda\mu}:=\Ii(\lambda_{\lambda\mu})
\in \Ss_\RR(\Omega_{U_{\lambda\mu}}^\HH)^\ast .
\]
Therefore
\[
\chi_\lambda=\mu_{\lambda\mu}*\chi_\mu .
\]
Since $\mu_{\lambda\mu}$ is slice preserving, it belongs to the center of
$(\Ss(\Omega_{U_{\lambda\mu}}^\HH),*)$, hence it commutes with both $f$ and $g$.
The cocycle identity is immediate from the corresponding identity for the maps
$\lambda_{\lambda\mu}$.
\end{proof}

\begin{remark}\label{rem:weak-not-strong}
Assume that $f$ and $g$ are weakly equivalent but not strongly equivalent, and let
\[
f=\chi^{-*}*g*\chi
\]
for some semiregular conjugator $\chi\in \Mm(\Omega)$ given by
Theorem~\ref{thm:weak-equivalence-quaternionic}.
Then $\chi$ cannot be regular and nowhere vanishing near a point of the base where
the multiplicities of $\cdiv(f)$ and $\cdiv(g)$ differ.
Indeed, if $\chi$ were regular and zero-free on some circular neighborhood
$\Omega_U^\HH$, then its restriction would give a local regular conjugacy between
$f_{|\Omega_U^\HH}$ and $g_{|\Omega_U^\HH}$.
By Proposition~\ref{prop:strong-local-conjugators} and the orbit-bundle interpretation
of strong equivalence, this would imply strong equivalence on $\Omega_U^\HH$, hence
equality of the local central divisors, a contradiction.

Therefore, if $f$ and $g$ are weakly but not strongly equivalent, every global
semiregular conjugator must necessarily develop singular behaviour over the locus where
the two central divisors do not match.
In other words, the mismatch of $\cdiv(f)$ and $\cdiv(g)$ is precisely the obstruction
to the existence of a local zero-free regular conjugator.

In the concrete example discussed later, one function has a spherical central divisor and
the other has trivial central divisor; there one sees explicitly that every conjugator must
degenerate on that sphere.
\end{remark}





\subsection{Examples}

We conclude the section with three examples illustrating:
\begin{itemize}
\item weak equivalence versus strong equivalence;
\item the role of the central divisor;
\item the difference between a global holomorphic section, a global conjugating function,
and arbitrary local conjugators.
\end{itemize}
\begin{remark}\label{rem:star-product-evaluation}
In the quaternionic case, the $*$-product admits the following pointwise evaluation formula.
If $a,b\in \Ss(\Omega)$ and $x\in\Omega$, then
\[
(a*b)(x)=
\begin{cases}
0, & \text{if } a(x)=0,\\[4pt]
a(x)\,b\!\bigl(a(x)^{-1}xa(x)\bigr), & \text{if } a(x)\neq 0.
\end{cases}
\]
Equivalently, on the set where $a$ does not vanish, one may write
\[
(a*b)(x)=a(x)\,b\bigl(T_a(x)\bigr),
\qquad
T_a(x):=a(x)^{-1}xa(x).
\]

The map $T_a$ preserves the real part and the norm, hence it sends each sphere
$\mathbb S_x$ onto itself. This formula will be used in the examples
below to pass from identities involving $*$-products to pointwise information on
zeros and conjugators.
\end{remark}

\begin{example}\label{ex:weak-not-strong}
Consider the slice regular polynomials on $\HH$
\[
Q(x):=\Delta_i(x)\,i=(x^2+1)i,
\qquad
R(x):=(x-i)*(x-i)j.
\]
Since \( (x-i)*(x-i)=x^2-2xi-1\),
we get
\[
R(x)=\bigl(x^2-2xi-1\bigr)j=(x^2-1)j-2xk.
\]

Both functions are purely vectorial, hence
\[
\Tr(Q)\equiv 0,\qquad \Tr(R)\equiv 0.
\]
Moreover,
\[
N(Q)=N(\Delta_i)\,N(i)=\Delta_i^2,
\qquad
N(R)=N\bigl((x-i)*(x-i)\bigr)\,N(j)=\Delta_i^2.
\]
Therefore $Q$ and $R$ are weakly equivalent.

However, they are not strongly equivalent.
Indeed, the vector part of $Q$ has only one nonzero slice-preserving component, namely
$x^2+1$, so
\[
\cdiv(Q)=[\ii]+[-\ii].
\]
On the other hand,
\(R(x)=(x^2-1)j-2xk\),
so the slice-preserving components of $R$ are
\( (0,\ x^2-1, -2x)\),
and they have no nonconstant common divisor. Hence
\[
\cdiv(R)=0.
\]

Thus $Q$ and $R$ are weakly equivalent but not strongly equivalent.

By Theorem~\ref{thm:weak-equivalence-quaternionic}, there exists a semiregular
function $\chi$ such that
\[
Q=\chi^{-*}*R*\chi .
\]
Equivalently, on the open set where $\chi$ and $\chi^{-*}$ are regular,
\[
\chi*Q=R*\chi.
\]
One can write down explicitly two nontrivial
solutions of the previous equation
Namely, 
the functions
\[
\chi_1:=Q+R
\qquad\text{and}\qquad
\chi_2:=Qi-iR
\]
both satisfy
\[
\chi_\nu*Q=R*\chi_\nu,
\qquad \nu=1,2.
\]

Indeed, since \(Q\) and \(R\) are purely vectorial and
\[
N(Q)=N(R)=\Delta_i^2,
\]
one has
\[
Q*Q=-N(Q)=-\Delta_i^2,
\qquad
R*R=-N(R)=-\Delta_i^2.
\]
Therefore
\[
(Q+R)*Q=Q*Q+R*Q=-\Delta_i^2+R*Q
\]
and
\[
R*(Q+R)=R*Q+R*R=R*Q-\Delta_i^2,
\]
so that
\[
\chi_1*Q=R*\chi_1.
\]

For the second solution, observe first that
\[
Qi=(x^2+1)i^2=-(x^2+1),
\]
while
\[
iR=i\bigl((x^2-1)j-2xk\bigr)=(x^2-1)k+2xj.
\]
Hence
\[
\chi_2=Qi-iR
=-(x^2+1)-2xj-(x^2-1)k.
\]
A direct computation gives
\begin{align*}
\chi_2*Q
&=
\bigl(-(x^2+1)-2xj-(x^2-1)k\bigr)*(x^2+1)i \\
&=
-(x^2+1)^2\,i-(x^2-1)(x^2+1)\,j+2x(x^2+1)\,k,
\end{align*}
whereas
\begin{align*}
R*\chi_2
&=
\bigl((x^2-1)j-2xk\bigr)*
\bigl(-(x^2+1)-2xj-(x^2-1)k\bigr) \\
&=
-(x^2+1)^2\,i-(x^2-1)(x^2+1)\,j+2x(x^2+1)\,k.
\end{align*}
Thus also
\[
\chi_2*Q=R*\chi_2.
\]

In particular, both \(\chi_1\) and \(\chi_2\) yield semiregular conjugators:
\[
Q=\chi_\nu^{-*}*R*\chi_\nu
\]
wherever \(\chi_\nu^{-*}\) is defined.

Moreover, \(\chi_1\) and \(\chi_2\) are linearly independent, since
\(\chi_1\) is purely vectorial whereas \(\chi_2\) has nonzero scalar part
\(- (x^2+1)\).

We now show that every such conjugator must have a pathology on the sphere
\(\mathbb S_i\), and in fact already at the point \(i\).
Let \(J\in\mathbb S_{\HH}\setminus\{i\}\).
Since \(J^2=-1\), one has
\[
R(J)=\bigl(J^2-2Ji-1\bigr)j=-2(1+Ji)j.
\]
In particular, \(R(J)\neq 0\) for every \(J\neq i\).
Set
\[
u:=1+Ji.
\]
A direct computation gives
\[
Ju=J+J^2i=J-i=-ui.
\]
Hence
\[
u^{-1}Ju=-i,
\]
and therefore
\[
T_R(J):=R(J)^{-1}JR(J)
=(uj)^{-1}J(uj)
=j^{-1}(u^{-1}Ju)j
=j^{-1}(-i)j
=i.
\]
So \(T_R(J)=i\) for every \(J\in\mathbb S_{\HH}\setminus\{i\}\).

Assume now that \(\chi\) is regular at \(i\).
Evaluating the identity \(\chi*Q=R*\chi\) at such a \(J\), and using the
evaluation formula for the $*$-product, we obtain
\[
0=(\chi*Q)(J)=(R*\chi)(J)=R(J)\,\chi\!\bigl(T_R(J)\bigr)=R(J)\chi(i).
\]
Since \(R(J)\neq 0\), it follows that
\[
\chi(i)=0.
\]
Therefore, if a conjugator is regular at \(i\), then it must vanish at \(i\);
equivalently, \(\chi^{-*}\) has a pole at \(i\).
If instead \(\chi\) is not regular at \(i\), then it already has a pole there.

In either case, every semiregular conjugator between \(Q\) and \(R\) fails to be
regular and invertible at \(i\). In particular, no conjugator can be regular and
nowhere vanishing on a circular neighborhood of the sphere \(\mathbb S_i\).
This gives a concrete manifestation of the obstruction produced by the mismatch
between \(\cdiv(Q)\) and \(\cdiv(R)\).
\end{example}

\begin{example}\label{ex:global-conjugating-function}
We now exhibit a situation in which strong equivalence is realized by a single global
zero-free conjugating function.

Let
\[
F(z):=z(i+j)-k,
\qquad
H(z):=z(i+j)+\frac{1}{\sqrt{2}}(i-j),
\qquad z\in D,
\]
and let
\(f:=\Ii(F), h:=\Ii(H)\).
Both $F$ and $H$ are intrinsic holomorphic stem functions.

Since both functions are purely vectorial, one has
\[
\Tr(f)=\Tr(h)=0.
\]
A direct computation gives
\[
N(F)(z)=2z^2+1,
\qquad
N(H)(z)=2z^2+1,
\]
hence \(N(f)\equiv N(h)\).
Moreover, the components of $F$ and $H$ have no common zero, so
\[
\cdiv(f)=\cdiv(h)=0.
\]
Therefore $f$ and $h$ are strongly equivalent.

Now set
\[
\alpha:=\frac{1}{\sqrt{2}}+\frac{1}{2}(i+j)\in\HH^\ast.
\]
Then \(N(\alpha)=1\), and a direct computation shows that
\[
F(z)\alpha=\alpha H(z)
\qquad \forall z\in D.
\]
Let \(\chi:=\Ii(\alpha)\), which is just the constant slice regular function with value
\(\alpha\). Since \(\chi\) is zero-free, the previous identity translates into
\[
f=\chi^{-*}*h*\chi
\qquad\text{on }\Omega.
\]

Thus, in this case, the holomorphic section of the orbit bundle is induced by a single
global zero-free slice regular conjugator.
\end{example}

The next example shows that, even in the strongly equivalent case, two arbitrary local
conjugators need not differ by a slice-preserving factor.
\begin{example}\label{ex:local-conjugators-centralizer}
Consider
\[
F(z):=z^2i+\sqrt{2}\,zj+k,
\qquad
H_1(z):=(z^2-1)i+2z\,j,
\qquad z\in D,
\]
and let
\[
f:=\Ii(F),\qquad h_1:=\Ii(H_1).
\]
Again, both functions are purely vectorial, so
\[
\Tr(f)=\Tr(h_1)=0.
\]
Moreover,
\[
N(F)(z)=z^4+2z^2+1=(z^2+1)^2,
\]
and
\[
N(H_1)(z)=(z^2-1)^2+(2z)^2=z^4+2z^2+1=(z^2+1)^2,
\]
hence \(N(f)=N(h_1)\).

The slice-preserving components of \(f\) are
\((z^2, \sqrt{2}\,z, 1)\),
while those of \(g_1\) are
\(z^2-1, 2z, 0\).
In both cases the common divisor is trivial, so
\[
\cdiv(f)=\cdiv(h_1)=0.
\]
Therefore \(f\) and \(h_1\) are strongly equivalent.

Now define
\begin{align*}
A_1(z)&:=F(z)+H_1(z)
=(2z^2-1)i+(\sqrt{2}+2)z\,j+k,\\
A_2(z)&:=z+\Bigl(1-\frac{\sqrt{2}}{2}\Bigr)(i-k).
\end{align*}
A direct computation shows that
\[
F(z)A_m(z)=A_m(z)H_1(z)
\qquad\text{for }m=1,2.
\]
Hence, wherever \(A_m\) is invertible, the induced slice function
\[
\chi_m:=\Ii(A_m)
\]
is a local conjugator between \(f\) and \(g_1\).

Set
\[
U_m:=D\setminus V\bigl(N(A_m)\bigr),\qquad m=1,2.
\]
Then \(\chi_m\) is zero-free on \(\Omega_{U_m}^\HH\), and
\[
f=\chi_m^{-*}*g_1*\chi_m
\qquad\text{on }\Omega_{U_m}^\HH.
\]

On the overlap \(\Omega_{U_1\cap U_2}^\HH\), consider
\[
h_{12}:=\chi_1^{-*}*\chi_2.
\]
Using the two conjugacy identities, one obtains
\[
h_{12}*g_1=g_1*h_{12}.
\]
Thus \(h_{12}\) belongs to the \( * \)-centralizer of \(g_1\).

This example illustrates Remark~\ref{rem:transition-centralizer}: although the special
transition functions arising from a chosen lifting of the orbit section are
slice-preserving and central, the ratio of two arbitrary local conjugators need only lie
in the \( * \)-centralizer, and in general there is no reason for it to be
slice-preserving.
\end{example}

\section{Models and semi-models}\label{sec:models}

In this section we introduce the notion of model for a slice regular function and study when such a
model exists and to what extent it is unique. We then introduce a more flexible notion of semi-model,
which is the natural object arising from the alignment procedure based on automorphisms and
$*$-conjugation.

We start with a preliminary result.

\begin{lemma}\label{lem:symmetrization-two-squares}
Let $\Omega=\Omega_D^A$ be a basic domain, let $f\in \Ss(\Omega)$, and fix
$I\in\mathbb S_A$.
Then there exists a function
\[
h\in \Ss_I(\Omega)
\]
such that
\(N(h)=N(f)\).
Equivalently, writing
\[
h=h_1+h_2I,
\qquad h_1,h_2\in \Ss_\RR(\Omega),
\]
one has
\[
N(f)=h_1^2+h_2^2.
\]

If, moreover, $f$ has no real or spherical zeros, then $h$ may be chosen in such a way that
it has no real or spherical zeros. Equivalently, all zeros of $h$ are isolated nonreal
and lie in the slice $\CC_I$.
\end{lemma}

\begin{proof}
If $N(f)\equiv 0$, the conclusion is trivial: it is enough to take $h\equiv 0$.
So assume $N(f)\not\equiv 0$.

Since $N(f)$ is slice preserving, there exists an intrinsic holomorphic function
\[
\Phi:D\to\CC
\]
such that
\[
N(f)=\Ii(\Phi).
\]

We first prove that every real zero of $\Phi$ has even order.
Write
\[
f=f_0+f_1e_1+\dots+f_ne_n,
\qquad f_\ell\in \Ss_\RR(\Omega),
\]
so that
\[
N(f)=f_0^2+\dots+f_n^2.
\]
Let $x_0\in D\cap\RR$ be such that $\Phi(x_0)=0$, equivalently $N(f)(x_0)=0$.
Then
\[
f_0(x_0)=\dots=f_n(x_0)=0.
\]
Let $F_\ell:D\to\CC$ be the intrinsic holomorphic stem inducing $f_\ell$.
For each $\ell$ we may write
\[
F_\ell(z)=(z-x_0)^{m_\ell}G_\ell(z),
\]
where either $F_\ell\equiv 0$ or $G_\ell(x_0)\neq 0$.
Set
\[
m:=\min\{m_\ell:\ \ell=0,\dots,n\}.
\]
Then
\[
\Phi(z)=F_0(z)^2+\dots+F_n(z)^2
=(z-x_0)^{2m}\bigl(G_0(z)^2+\dots+G_n(z)^2\bigr).
\]
Since $x_0\in\RR$ and each $G_\ell$ is intrinsic, all the values $G_\ell(x_0)$ are real.
Moreover, at least one of the $G_\ell(x_0)$ with $m_\ell=m$ is nonzero, hence
\[
G_0(x_0)^2+\dots+G_n(x_0)^2>0.
\]
Therefore $\ord_{x_0}\Phi=2m$ is even.

It follows that $\Phi$ is intrinsic holomorphic, nonnegative on $D\cap\RR$, and all its
real zeros have even order. By \cite[Theorem~3.2]{AltavillaAMPA}, there exists a
one-slice-preserving function
\[
g\in \Ss_J(\Omega)
\]
for some $J\in\mathbb S_A$ such that
\(N(g)=N(f)\).
Write
\[
g=a+bJ,
\qquad a,b\in \Ss_\RR(\Omega).
\]
Now define
\[
h:=a+bI\in \Ss_I(\Omega).
\]
Then
\[
N(h)=a^2+b^2=N(g)=N(f).
\]
This proves the first part of the statement. Writing $h=h_1+h_2I$ with
$h_1:=a$ and $h_2:=b$, we also obtain
\[
N(f)=h_1^2+h_2^2.
\]

Assume now, in addition, that $f$ has no real or spherical zeros.
Then $N(f)$ has no real zeros. Hence the divisor of $\Phi$ has the form
\[
\operatorname{div}(\Phi)=\sum_\lambda m_\lambda\bigl([z_\lambda]+[\bar z_\lambda]\bigr),
\qquad
z_\lambda=\alpha_\lambda+\ii\beta_\lambda\in D\cap\CC^+,
\]
with $\beta_\lambda>0$.
Fix the chosen imaginary unit $I$, and for every $\lambda$ set
\[
q_\lambda:=\alpha_\lambda+\beta_\lambda I\in \Omega\cap\CC_I^+.
\]

By Weierstrass factorization on the simply connected base $D$, there exists a zero-free
intrinsic holomorphic function
\[
S:D\to\CC
\]
such that
\[
\Phi=S\prod_\lambda \Delta_{z_\lambda}^{\,m_\lambda},
\]
where
\[
\Delta_{z_\lambda}(z)=(z-z_\lambda)(z-\bar z_\lambda).
\]
Passing to slice functions, this means that there exists a zero-free slice-preserving
function
\[
s:=\Ii(S)\in \Ss_\RR(\Omega)
\]
such that
\[
N(f)=s\prod_\lambda \Delta_{q_\lambda}^{\,m_\lambda}.
\]

Since $S$ is zero-free on the simply connected domain $D$ and is intrinsic, it admits an
intrinsic holomorphic square root. Therefore $s$ admits a slice-preserving square root
\[
\sigma\in \Ss_\RR(\Omega),
\qquad \sigma^2=s.
\]

Now define
\[
h:=\sigma * \prod_\lambda (x-q_\lambda)^{*m_\lambda}.
\]
Since all factors take values in the same slice $\CC_I$, the function $h$ belongs to
$\Ss_I(\Omega)$. Moreover,
\[
N(h)=N(\sigma)\prod_\lambda N(x-q_\lambda)^{m_\lambda}
=\sigma^2\prod_\lambda \Delta_{q_\lambda}^{\,m_\lambda}
=s\prod_\lambda \Delta_{q_\lambda}^{\,m_\lambda}
=N(f).
\]

By construction, on each sphere $\mathbb S_{q_\lambda}$ the function $h$ has exactly one
isolated zero, namely $q_\lambda$, of isolated multiplicity $m_\lambda$. In particular,
$h$ has no real or spherical zeros.

This concludes the proof.
\end{proof}

\begin{remark}\label{rem:two-squares-not-canonical}
The functions $h_1,h_2$ given by Lemma~\ref{lem:symmetrization-two-squares} are in
general highly non-unique and should not be confused with the slice-preserving
components of $f$.

Indeed, the lemma only uses the slice-preserving function $N(f)$, not the original
function $f$ itself. In particular, different slice regular functions with the same
symmetrization lead to the same existence statement for $h_1,h_2$.

Moreover, even for a fixed $f$, the pair $(h_1,h_2)$ is far from canonical.
When $\Omega\cap\RR\neq\emptyset$, the construction depends on the choice of a
one-slice-preserving function $g$ with $N(g)=N(f)$.
When $\Omega\cap\RR=\emptyset$, the construction depends on an arbitrary decomposition
of the inducing holomorphic function $\Phi$ on $D^+$ as a sum of two squares.
Thus the pair $(h_1,h_2)$ is an auxiliary device attached to $N(f)$, rather than a
structure intrinsically encoded in $f$.
\end{remark}

\begin{example}\label{ex:two-squares-nonunique}
Let
\(\Omega:=\{x\in\HH:\ |x|<1\}\)
and consider the constant slice regular function
\(f(x)\equiv i\).
Then\(
N(f)\equiv 1\).

A first possible choice in Lemma~\ref{lem:symmetrization-two-squares} is simply
\[
h_1(x)\equiv 1,
\qquad
h_2(x)\equiv 0.
\]
Indeed,
\[
N(f)=1=h_1^2+h_2^2.
\]

However, a completely different choice is
\[
h_1(x):=\frac{1-x^2}{1+x^2},
\qquad
h_2(x):=\frac{2x}{1+x^2}.
\]
Since $1+x^2$ has no zeros on $\Omega$, both $h_1$ and $h_2$ are slice-preserving
regular functions on $\Omega$, and
\[
h_1(x)^2+h_2(x)^2
=
\frac{(1-x^2)^2+4x^2}{(1+x^2)^2}
=
\frac{(1+x^2)^2}{(1+x^2)^2}
=
1
=
N(f).
\]

Thus even in the simplest possible case the pair $(h_1,h_2)$ is not unique and has
no direct relation with the original function $f$.
\end{example}

\subsection{Models and existence results}

\begin{definition}\label{def:model}
Let \(f\in \Ss(\Omega)\).
A \emph{model} for \(f\) is a slice regular function \(g\in \Ss(\Omega)\) such that:
\begin{enumerate}
\item \(g\) is strongly equivalent to \(f\);
\item all isolated nonreal zeros of \(g\) lie in a single slice \(\CC_I\), for some \(I\in\mathbb S_A\).
\end{enumerate}
\end{definition}

\begin{remark}
A model is therefore an aligned representative inside the strong-equivalence class of \(f\).
\end{remark}

If \(f\) is slice preserving or $\CC_I$-preserving for some $I\in\SF_A$, then \(f\) is itself a model. 

\begin{remark}\label{rem:div-of-slicepres}
If \(h\in \Ss_\RR(\Omega)\), we denote by
\[
\operatorname{div}(h)
\]
the divisor of the intrinsic holomorphic stem inducing \(h\) on the base \(D=D_\Omega\).

Moreover, if \(I\in\mathbb S_A\) and
\[
g=f_0+hI
\]
is a \(\CC_I\)-preserving function with \(h\not\equiv 0\), then the vector part of \(g\) is
\(hI\), hence
\[
\cdiv(g)=\operatorname{div}(h).
\]
\end{remark}

\begin{theorem}[Criterion for one-slice-preserving models]
\label{thm:criterion-one-slice-model}
Assume that \(f=f_0+f_v\in \Ss(\Omega)\) is not slice preserving, and fix \(I\in\mathbb S_A\).
Then the following are equivalent:
\begin{enumerate}
\item \(f\) is strongly equivalent to a \(\CC_I\)-preserving function;
\item there exists \(h\in \Ss_\RR(\Omega)\), \(h\not\equiv 0\), such that
\[
h^2=N(f_v)
\qquad\text{and}\qquad
\operatorname{div}(h)=\cdiv(f).
\]
\end{enumerate}
In this case, the function
\[
g:=f_0+hI
\]
is \(\CC_I\)-preserving, strongly equivalent to \(f\), and therefore is a model for \(f\).
\end{theorem}
\begin{proof}
Assume first that \(f\) is strongly equivalent to a \(\CC_I\)-preserving function \(g\).
Then \(g\) has the form
\[
g=g_0+hI
\]
for some \(g_0,h\in \Ss_\RR(\Omega)\).
Since \(f\) is not slice preserving, also \(g\) cannot be slice preserving, hence
\(h\not\equiv 0\).

Because \(f\) and \(g\) are strongly equivalent, they have the same trace and the same norm.
From \(2f_0=\Tr(f)=\Tr(g)=2g_0\),
we obtain
\(g_0=f_0\).
Moreover,
\[
N(g)=g_0^2+h^2=f_0^2+h^2,
\]
while
\[
N(f)=f_0^2+N(f_v).
\]
Hence \(N(g)=N(f)\) implies
\(h^2=N(f_v)\).
Finally, since \(g\) is not slice preserving, strong equivalence also gives equality of the
central divisors:
\[
\cdiv(f)=\cdiv(g).
\]
By Remark~\ref{rem:div-of-slicepres}, this means
\[
\cdiv(f)=\operatorname{div}(h).
\]

Conversely, assume that there exists \(h\in \Ss_\RR(\Omega)\), \(h\not\equiv 0\), such that
\[
h^2=N(f_v)
\qquad\text{and}\qquad
\operatorname{div}(h)=\cdiv(f).
\]
Define
\[
g:=f_0+hI.
\]
Then \(g\) is \(\CC_I\)-preserving and, since \(h\not\equiv 0\), it is not slice preserving.

Its trace is
\(\Tr(g)=2f_0=\Tr(f)\),
and its norm is \(N(g)=f_0^2+h^2=f_0^2+N(f_v)=N(f)\).
Again by Remark~\ref{rem:div-of-slicepres},
\[
\cdiv(g)=\operatorname{div}(h)=\cdiv(f).
\]
Therefore \(f\) and \(g\) have the same invariants \(\Tr\), \(N\), and \(\cdiv\), so they are
strongly equivalent.

Since \(g\) is \(\CC_I\)-preserving, all its isolated nonreal zeros lie in the slice \(\CC_I\).
Hence \(g\) is a model for \(f\).
\end{proof}

Therefore, we have obtained that
if \(f\) is strongly equivalent to a \(\CC_I\)-preserving function, then \(f\) admits a model.
Another large class of slice regular functions admitting a model is the following.

\begin{theorem}[Purely vectorial model theorem]
\label{thm:purely-vectorial-model}
Every purely vectorial slice regular function \(u\in \Ss(\Omega)\) admits a model.
\end{theorem}

\begin{proof}
If \(u\equiv 0\), then \(u\) is itself a model. So assume \(u\not\equiv 0\).

Write
\[
u=u_1e_1+\cdots+u_ne_n,
\qquad u_\ell\in \Ss_\RR(\Omega),
\]
and let
\[
U_\ell:D\to\CC
\]
be the intrinsic holomorphic stem inducing \(u_\ell\).
By definition of \(\cdiv(u)\), there exists an intrinsic holomorphic function
\[
\delta:D\to\CC
\]
whose divisor is \(\cdiv(u)\), and intrinsic holomorphic functions
\[
V_1,\dots,V_n:D\to\CC
\]
such that
\[
U_\ell=\delta\,V_\ell,
\qquad \ell=1,\dots,n,
\]
and such that the curve
\[
V:=(V_1,\dots,V_n)
\]
has trivial divisor. Let
\[
d:=\Ii(\delta)\in \Ss_\RR(\Omega),
\qquad
v:=\Ii(V_1)e_1+\cdots+\Ii(V_n)e_n\in \Ss(\Omega).
\]
Then \(v\) is purely vectorial,
\[
u=d\,v,
\]
and
\[
\cdiv(v)=0.
\]

Since \(v\) is purely vectorial and \(\cdiv(v)=0\), it has no real or spherical zeros.
Fix the imaginary unit \(e_1\). By Lemma~\ref{lem:symmetrization-two-squares}, there exists
a function
\[
h=h_1+h_2e_1\in \Ss_{e_1}(\Omega)
\]
such that
\[
N(h)=N(v),
\]
and, since \(v\) has no real or spherical zeros, \(h\) may be chosen so as to have no
real or spherical zeros.

Now define
\[
\widetilde v:=h\,e_2=h_1e_2+h_2e_3.
\]
Then \(\widetilde v\) is again purely vectorial. Moreover, since right multiplication by the
nonzero constant \(e_2\) preserves the zero set, \(\widetilde v\) has exactly the same zeros
as \(h\). In particular, \(\widetilde v\) has no real or spherical zeros, and all its isolated
nonreal zeros lie in the slice \(\CC_{e_1}\), because the same is true for \(h\).

Furthermore,
\[
N(\widetilde v)=N(h)\,N(e_2)=N(h)=N(v).
\]
Since \(\widetilde v\) has no real or spherical zeros, its central divisor is trivial:
\[
\cdiv(\widetilde v)=0=\cdiv(v).
\]

Now set
\[
g:=d\,\widetilde v.
\]
Since \(d\) is slice preserving, \(g\) is again purely vectorial. Also,
\[
\Tr(g)=0=\Tr(u),
\]
and
\[
N(g)=d^2\,N(\widetilde v)=d^2\,N(v)=N(u).
\]

At the level of vector stems, the components of \(g\) are obtained by multiplying those of
\(\widetilde v\) by the common factor \(\delta\). Since \(\cdiv(\widetilde v)=0\), it follows that
\[
\cdiv(g)=\operatorname{div}(\delta)=\cdiv(u).
\]
Therefore \(g\) and \(u\) have the same invariants \(\Tr\), \(N\), and \(\cdiv\), hence they are
strongly equivalent.

Finally, multiplication by the slice-preserving factor \(d\) can only add real or spherical zeros.
Therefore the isolated nonreal zeros of \(g\) are exactly those of \(\widetilde v\), and hence they
all lie in the same slice \(\CC_{e_1}\). Thus \(g\) is a model for \(u\).
\end{proof}

\begin{corollary}\label{cor:vector-part-model}
Let \(f=f_0+f_v\in \Ss(\Omega)\). Then the vector part \(f_v\) admits a model.
\end{corollary}

\begin{proof}
Apply Theorem~\ref{thm:purely-vectorial-model} to the purely vectorial function \(f_v\).
\end{proof}

\begin{example}\label{ex:octonionic-model-nonzero-trace}
Let
\(D:=\{z\in\CC:\ |z|<1\}\), and \(\Omega:=\Omega_D^\OO\).
Consider the octonionic slice regular polynomial
\[
f(x):=x^2+x^4e_1+2x^3e_2+\sqrt6\,x^2e_3+2xe_4+e_5.
\]
Its scalar part is \(f_0(x)=x^2\), so
\[
\Tr(f)=2x^2\not\equiv 0.
\]
Moreover, its vector part is
\[
f_v(x)=x^4e_1+2x^3e_2+\sqrt6\,x^2e_3+2xe_4+e_5,
\]
hence
\[
N(f_v)=x^8+4x^6+6x^4+4x^2+1=(x^2+1)^4.
\]
Since \(z^2+1\) has no zeros in \(D\), the slice-preserving function
\[
h(x):=(x^2+1)^2
\]
is zero-free on \(\Omega\). Define
\[
g(x):=x^2+h(x)e_1=x^2+(x^2+1)^2e_1.
\]
Then \(g\) is \(\CC_{e_1}\)-preserving.

We now compare the invariants of \(f\) and \(g\). First,
\[
\Tr(g)=2x^2=\Tr(f).
\]
Next,
\[
N(g)=x^4+h(x)^2=x^4+(x^2+1)^4=N(f).
\]
Finally, since the vector components of \(f\) are
\((x^4, 2x^3, \sqrt6\,x^2, 2x, 1)\),
their common divisor on \(D\) is trivial, so
\[
\cdiv(f)=0.
\]
On the other hand, the vector part of \(g\) is \(h(x)e_1\), and \(h\) is zero-free on \(D\), hence
\[
\cdiv(g)=0.
\]
Therefore \(f\) and \(g\) are strongly equivalent. Since \(g\) is \(\CC_{e_1}\)-preserving,
it is a model for \(f\).

At the level of stem functions, if
\[
F(z):=z^2+z^4e_1+2z^3e_2+\sqrt6\,z^2e_3+2ze_4+e_5,
\qquad
G(z):=z^2+(z^2+1)^2e_1,
\]
then Theorem~\ref{thm:strong-equivalence-sections} yields an intrinsic holomorphic map
\[
\phi:D\to\Aut(\OO_\CC)
\]
such that
\[
G(z)=\phi(z)\bigl(F(z)\bigr)
\qquad \forall z\in D.
\]
Notice that \(\phi\) cannot be chosen constant, since the vector part of \(F(z)\) is not contained
in a fixed complex line of \(\Im(\OO_\CC)\): for instance,
\[
F_v(0)=e_5,
\qquad
F_v(1/2)=\tfrac1{16}e_1+\tfrac14 e_2+\tfrac{\sqrt6}{4}e_3+e_4+e_5
\]
are not proportional. For an explicit computation of \(\phi\) see Appendix~\ref{G2}.
\end{example}

\begin{example}\label{ex:octonionic-two-components}
Consider the purely vectorial octonionic slice regular polynomial
\[
f(x):=x^4e_1+2x^3e_2+2xe_3+e_4.
\]
Then
\(\Tr(f)\equiv 0\) and
\(N(f)=x^8+4x^6+4x^2+1\).
A direct computation shows that
\[
x^8+4x^6+4x^2+1
=
\bigl(x^4+(\sqrt{10}-2)x^2+1\bigr)^2
+
(8-2\sqrt{10})(x^3-x)^2.
\]
Set
\[
h_1(x):=x^4+(\sqrt{10}-2)x^2+1,
\qquad
h_2(x):=\sqrt{\,8-2\sqrt{10}\,}\,(x^3-x),
\]
and define
\[
g(x):=h_1(x)e_1+h_2(x)e_2.
\]
Then \(g\) is purely vectorial and has only two nonzero slice-preserving components. Moreover,
\[
\Tr(g)=0=\Tr(f),
\qquad
N(g)=h_1^2+h_2^2=N(f).
\]

We now compare the central divisors.
Since the components of \(f\) are
\((x^4, 2x^3, 2x, 1)\),
they have no common zero, hence
\[
\cdiv(f)=0.
\]
On the other hand, the components of \(g\) are \(h_1\) and \(h_2\), and these have no common zero:
indeed \(h_2\) vanishes only at \(0,\pm1\), while
\[
h_1(0)=1,\qquad h_1(1)=\sqrt{10},\qquad h_1(-1)=\sqrt{10}.
\]
Therefore
\[
\cdiv(g)=0.
\]
Hence \(f\) and \(g\) are strongly equivalent.

In particular, this gives an explicit reduction of \(f\) to a strongly equivalent purely vectorial
slice regular function with only two components.
\end{example}

\begin{remark}
In both examples the restriction to a suitable basic domain is essential.
Indeed, on the whole complex plane the corresponding \(\CC_{e_1}\)-preserving representatives
would have nontrivial central divisor, whereas the original functions have \(\cdiv=0\).
Thus the strong equivalence with a one-slice-preserving model is a genuinely local phenomenon.
\end{remark}

We now show that a model does not always exist inside the strong-equivalence class.
\begin{example}\label{ex:no-CI-preserving-representative}
Consider the quaternionic slice regular polynomial
\[
f(x):=(x-i)*(x-2j)*(x-3k).
\]
Since we are in the quaternionic case, the $*$-product is associative, and a direct expansion gives
\[
f(x)
=
x^3-x^2(i+2j+3k)+x(6i-3j+2k)+6.
\]
Hence
\[
f_0(x)=x^3+6,
\qquad
\Tr(f)=2x^3+12\not\equiv 0,
\]
and
\[
f_v(x)=x(6-x)i-x(2x+3)j+x(2-3x)k.
\]

By repeated application of the camshaft effect~\cite{Ghiloni2010}, the zero set of \(f\) is
\[
\Zz(f)
=
\left\{
i,\,
\frac85\,i+\frac65\,j,\,
\frac95\,i+\frac{144}{65}\,j+\frac{12}{13}\,k
\right\}.
\]
While the quaternion $i$ is easily computed as zero of $f$, the explicit computation of the other two zeros is done in the following Appendix~\ref{appendix:computation}.
In particular, \(f\) has three isolated nonreal zeros, lying on the spheres of radii
\(1,2,3\), respectively.

We now show that \(f\) is \emph{not} strongly equivalent to any \(\CC_I\)-preserving function.

Indeed, from the expression of the vector part we immediately get
\[
\cdiv(f)=[0],
\]
because the three slice-preserving components of \(f_v\) have the common factor \(x\), and no larger common divisor.

Moreover,
\begin{align*}
N(f_v)
&=
x^2(6-x)^2+x^2(2x+3)^2+x^2(2-3x)^2\\
&=
x^2\bigl((6-x)^2+(2x+3)^2+(2-3x)^2\bigr)\\
&=
x^2(14x^2-12x+49).
\end{align*}
Equivalently,
\[
N(f_v)=14\,x^2\,\Delta_q(x),
\qquad
q=\frac37+\frac{5\sqrt{26}}{14}\,I,
\]
for any \(I\in\mathbb S_{\HH}\). Thus \(N(f_v)\) has a real zero of multiplicity \(2\) at \(0\),
and one spherical zero of multiplicity \(2\).

Assume by contradiction that \(f\) is strongly equivalent to a \(\CC_I\)-preserving function.
Then, by Theorem~\ref{thm:criterion-one-slice-model}, there should exist
\(h\in \Ss_\RR(\Omega)\), \(h\not\equiv 0\), such that
\[
h^2=N(f_v)
\qquad\text{and}\qquad
\operatorname{div}(h)=\cdiv(f).
\]
However, this is impossible. Indeed, by the square-root criterion for slice-preserving
functions on simply connected circular domains, a slice-preserving function admits a square root
only if its real zeros have even isolated multiplicity and its spherical zeros have multiplicity
multiple of \(4\); see Proposition~3.1 in \cite{AltavillaPAMS}. Since \(N(f_v)\)
has a spherical zero of multiplicity \(2\), it cannot admit any slice-preserving square root.

Therefore \(f\) is not strongly equivalent to any \(\CC_I\)-preserving function.
In particular, the criterion of Theorem~\ref{thm:criterion-one-slice-model} does not apply to \(f\).
\end{example}

\begin{remark}\label{rem:vector-part-still-has-model}
The previous example shows that, even when \(\Tr(f)\not\equiv 0\), one cannot in general expect
\(f\) to be strongly equivalent to a \(\CC_I\)-preserving function.

On the other hand, this does not contradict Theorem~\ref{thm:purely-vectorial-model}.
Indeed, the obstruction only concerns the \emph{full} function \(f=f_0+f_v\). The purely vectorial part
\[
f_v(x)=x(6-x)i-x(2x+3)j+x(2-3x)k
\]
still admits a model by Theorem~\ref{thm:purely-vectorial-model}. In other words, even though
one cannot reduce \(f\) itself to a one-slice-preserving representative, one can still align the
isolated nonreal zeros of its vector part.
\end{remark}

A uniqueness statement for arbitrary models seems delicate. However, if the
models are themselves one-slice-preserving, then uniqueness holds up to the natural
action of constant automorphisms of the algebra.

\begin{proposition}\label{thm:uniqueness-model}
Let \(\Omega=\Omega_D^A\) be a basic domain and let \(f\in \Ss(\Omega)\).
Assume that \(g_1,g_2\in \Ss(\Omega)\) are two models for \(f\), and that
\(g_\ell\) is \(\CC_{I_\ell}\)-preserving and not slice preserving, for
\(\ell=1,2\).

Then there exist \(g_0,h\in \Ss_\RR(\Omega)\), with \(h\not\equiv 0\), and
\(\varepsilon\in\{\pm1\}\) such that
\[
g_1=g_0+hI_1,
\qquad
g_2=g_0+\varepsilon h I_2.
\]
In particular, there exists a constant automorphism \(\psi\in\Aut(A)\) such that
\[
\psi(I_1)=\varepsilon I_2
\qquad\text{and}\qquad
g_2=\psi\cdot g_1.
\]
Thus one-slice-preserving models are unique up to the natural action of
\(\Aut(A)\).
\end{proposition}

\begin{proof}
Since \(g_1\) and \(g_2\) are both models for \(f\), they are strongly equivalent to
\(f\), hence strongly equivalent to each other.

Because \(g_\ell\) is \(\CC_{I_\ell}\)-preserving and not slice preserving, one can write
\[
g_\ell=a_\ell+h_\ell I_\ell,
\qquad
a_\ell,h_\ell\in \Ss_\RR(\Omega),
\quad h_\ell\not\equiv 0.
\]
Equality of the traces gives
\[
2a_1=\Tr(g_1)=\Tr(g_2)=2a_2,
\]
hence \(a_1=a_2=:g_0\).

Equality of the norms yields
\(g_0^2+h_1^2=N(g_1)=N(g_2)=g_0^2+h_2^2\),
so that
\[
h_1^2=h_2^2.
\]
Moreover, by Remark~\ref{rem:div-of-slicepres},
\[
\operatorname{div}(h_1)=\cdiv(g_1)=\cdiv(g_2)=\operatorname{div}(h_2).
\]
Therefore the quotient of the inducing intrinsic holomorphic stems of \(h_2\) and
\(h_1\) extends to a nowhere vanishing intrinsic holomorphic function whose square is
identically \(1\). Since \(D\) is connected, this quotient must be constant equal to
\(\varepsilon\in\{\pm1\}\). Hence
\[
h_2=\varepsilon h_1.
\]
Setting \(h:=h_1\), we obtain
\[
g_1=g_0+hI_1,
\qquad
g_2=g_0+\varepsilon h I_2.
\]

Finally, choose \(\psi\in\Aut(A)\) such that \(\psi(I_1)=\varepsilon I_2\). Since
\(\psi\) fixes the real line pointwise, it fixes every slice-preserving function, and
therefore
\[
\psi\cdot g_1
=
\psi(g_0+hI_1)
=
g_0+h\,\psi(I_1)
=
g_0+\varepsilon h I_2
=
g_2.
\]
\end{proof}

\begin{remark}
The previous proposition shows uniqueness only for models which are themselves
one-slice-preserving. For a general model, the condition that all isolated nonreal zeros
lie in a single slice does not seem to force the whole function to be
\(\CC_I\)-preserving, so a corresponding uniqueness statement is less clear.
\end{remark}

\subsection{Semi-models}

We now introduce the more flexible notion which always arises from the general alignment procedure.

\begin{definition}\label{def:semimodel}
Let $f\in \Ss(\Omega)$ with $N(f)\not\equiv 0$.
A \emph{semi-model} for $f$ is a slice regular function $\widehat f\in \Ss(\Omega)$ such that:
\begin{enumerate}
\item $N(\widehat f)=N(f)$;
\item the spherical zeros of $\widehat f$ coincide with those of $f$, with the same spherical multiplicities;
\item all isolated nonreal zeros of $\widehat f$ lie in a single slice $\CC_I$, for some $I\in\mathbb S_A$.
\end{enumerate}
\end{definition}

\begin{remark}
Since $N(\widehat f)=N(f)$, the real zeros of $\widehat f$ and $f$ automatically coincide,
with the same multiplicities. The additional condition on spherical zeros is needed because
the symmetrization alone does not distinguish spherical zeros from isolated nonreal zeros
lying on the same sphere.
\end{remark}

\begin{remark}
Unlike a model, a semi-model need not be strongly equivalent to the original function.
It does, however, preserve the relevant zero data.
\end{remark}

First, we show that every slice regular function is strongly equivalent, both in the quaternionic and in the octonionic setting, to a function whose stem function takes values in \(\Span\{1,i,j\}\).

\begin{proposition}\label{prop:two-directions}
Let $\Omega=\Omega_D^A$ be a basic domain and let $f=f_0+f_v=\Ii(F)\in \Ss(\Omega)$.
Then there exist:
\begin{itemize}
\item a fixed quaternionic subalgebra
\[
\HH_0=\Span_{\RR}\{1,i,j,k\}\subset A,
\qquad k:=ij,
\]
\item slice-preserving functions $h_1,h_2\in \Ss_\RR(\Omega)$,
\end{itemize}
such that the slice regular function
\[
f':=f_0+h_1 i+h_2 j
\]
is strongly equivalent to $f$.
\end{proposition}

In particular, if \(F\) and \(F'\) denote the stem function of \(f\) and \(f'\), respectively, then by Theorem~\ref{thm:strong-equivalence-sections}, there exists an intrinsic
holomorphic map
\(\phi:D\to \Aut(A_\CC)\)
such that
\[
F'(z)=\phi(z)\bigl(F(z)\bigr)\subset \CC\oplus \CC i\oplus \CC j\subset (\HH_0)_\CC\subset A_\CC
\qquad \forall z\in D.
\]

\begin{proof}[proof of Proposition~\ref{prop:two-directions}]
Write
\[
f=f_0+f_v,
\qquad
f_v=\Ii(F_v).
\]
If $f_v\equiv 0$, then $f$ is slice preserving and the conclusion is trivial by taking
\(f'=f\) and \(h_1=h_2\equiv 0\).

Assume now that $f_v\not\equiv 0$.
As in the proof of Theorem~\ref{thm:purely-vectorial-model}, we factor out the
central divisor of the vector part. More precisely, there exists
\[
d\in \Ss_\RR(\Omega)
\]
such that
\[
\operatorname{div}(d)=\cdiv(f),
\]
and there exists a purely vectorial function
\(v\in \Ss(\Omega)\) with
\[
f_v=d\,v
\qquad\text{and}\qquad
\cdiv(v)=0.
\]
Therefore, $v$ has no real or spherical zeros.

We now apply to $v$ the same reduction argument used for purely vectorial functions.
By Lemma~\ref{lem:symmetrization-two-squares}, the slice-preserving function $N(v)$ can
be written as a sum of two squares. Choosing this decomposition as in the construction
behind Theorem~\ref{thm:purely-vectorial-model}, we obtain slice-preserving
functions
\[
a,b\in \Ss_\RR(\Omega)
\]
such that the purely vectorial function
\[
u:=a\,i+b\,j
\]
satisfies
\[
N(u)=N(v)
\qquad\text{and}\qquad
\cdiv(u)=0.
\]
In particular, $u$ takes values in the fixed quaternionic subalgebra
\[
\HH_0=\Span_{\RR}\{1,i,j,k\}\subset A.
\]

Now define
\[
f':=f_0+d\,u.
\]
Then
\[
f'=f_0+h_1 i+h_2 j,
\qquad
h_1:=da,\quad h_2:=db,
\]
so $f'$ has the required form.

We compare the invariants of $f$ and $f'$.
First,
\[
\Tr(f')=2f_0=\Tr(f).
\]
Next,
\[
N(f')
=
f_0^2+d^2N(u)
=
f_0^2+d^2N(v)
=
f_0^2+N(f_v)
=
N(f).
\]
Finally, since $d$ is slice preserving and $\cdiv(u)=0$, we get
\[
\cdiv(f')
=
\operatorname{div}(d)+\cdiv(u)
=
\cdiv(f)+0
=
\cdiv(f).
\]
Hence $f$ and $f'$ have the same three invariants $\Tr$, $N$ and $\cdiv$; by
Definition~\ref{def:weak-strong}, they are strongly equivalent.
\end{proof}

The equivalent formulation in terms of the stem is immediate: if $F'$ is the stem
function of $f'$, then from
\[
f'=f_0+h_1 i+h_2 j
\]
it follows that
\(F'(D)\subset \CC\oplus \CC i\oplus \CC j\subset (\HH_0)_\CC\).

Finally we prove our main theorem.
\begin{theorem}[Existence of semi-models]
\label{thm:semimodel-existence}
Let $\Omega=\Omega_D^A$ be a basic domain and let $f\in \Ss(\Omega)$ with $N(f)\not\equiv 0$.
Then $f$ admits a semi-model $\widehat f$.
\end{theorem}

\begin{proof}
By Proposition~\ref{prop:two-directions}, there exists a function
\[
f'=f_0'+h_1 i+h_2 j
\]
which is strongly equivalent to $f$.
In particular,
\[
\Tr(f')=\Tr(f),\qquad N(f')=N(f),\qquad \cdiv(f')=\cdiv(f).
\]
Hence $f$ and $f'$ have the same real and spherical zeros, with the same multiplicities.

Set
\[
v:=f'*k.
\]
Since $k$ is a nonzero constant, $v$ is slice regular and
\[
\Zz(v)=\Zz(f').
\]
Moreover, using $ik=-j$ and $jk=i$, we get
\[
v=f_0'k-h_1j+h_2i,
\]
so $v$ is purely vectorial and takes values in the fixed quaternionic subalgebra $\HH_0$.
Also,
\[
N(v)=N(f')=N(f).
\]
Thus $v$ has the same real and spherical zeros as $f$.

Now factor out the central divisor of $v$. More precisely, choose
\[
d\in \Ss_\RR(\Omega)
\]
such that
\[
\operatorname{div}(d)=\cdiv(v),
\]
and write
\[
v=d\,w,
\]
where $w$ is purely vectorial and satisfies
\[
\cdiv(w)=0.
\]
Hence $w$ has no real or spherical zeros.

Applying Theorem~\ref{thm:purely-vectorial-model} to $w$, we obtain a purely vectorial
slice regular function
\[
\widetilde w\in \Ss(\Omega)
\]
such that
\[
N(\widetilde w)=N(w),
\]
all isolated nonreal zeros of $\widetilde w$ lie in a single slice \(\CC_I\), and the moduli
and isolated multiplicities of such zeros are preserved.

Define
\[
\widehat f:=d\,\widetilde w.
\]
Since $d$ is slice preserving, we have
\[
N(\widehat f)=d^2N(\widetilde w)=d^2N(w)=N(v)=N(f).
\]
Moreover, multiplication by the slice-preserving factor $d$ only produces real or spherical zeros.
Therefore the isolated nonreal zeros of $\widehat f$ are exactly those of $\widetilde w$,
and hence they all lie in the same slice \(\CC_I\), with the same moduli and isolated multiplicities
as those of $w$, and therefore as those of $v$ and of $f$.

Finally, the real and spherical zeros of \(\widehat f\) are exactly those encoded by the factor \(d\),
hence they coincide with those of \(v\), and therefore with those of \(f\), with the same multiplicities.
This proves that \(\widehat f\) is a semi-model for \(f\).
\end{proof}

\begin{example}\label{ex:quaternionic-semimodel-explicit}
We reconsider the polynomial of Example~\ref{ex:no-CI-preserving-representative},
namely
\[
f(x):=(x-i)*(x-2j)*(x-3k)
=
x^3-x^2(i+2j+3k)+x(6i-3j+2k)+6.
\]
As already observed there, \(f\) is not strongly equivalent to any
\(\CC_I\)-preserving function.

We now make Proposition~\ref{prop:two-directions} explicit on the whole domain
\(\Omega=\HH\).

The scalar part of \(f\) is
\[
f_0(x)=x^3+6,
\]
while the vector part is
\[
f_v(x)=x(6-x)i-x(2x+3)j+x(2-3x)k.
\]
Hence
\begin{align*}
N(f_v)
&=
x^2(6-x)^2+x^2(2x+3)^2+x^2(2-3x)^2\\
&=
x^2(14x^2-12x+49).
\end{align*}
Completing the square,
\[
14x^2-12x+49
=
14\left(x-\frac37\right)^2+\frac{325}{7},
\]
so that
\[
N(f_v)
=
\left(\sqrt{14}\,x\left(x-\frac37\right)\right)^2
+
\left(\frac{5\sqrt{91}}7\,x\right)^2.
\]
Therefore we may choose the globally defined slice-preserving polynomials
\[
h_1(x):=\sqrt{14}\,x\left(x-\frac37\right),
\qquad
h_2(x):=\frac{5\sqrt{91}}7\,x.
\]
They satisfy
\[
h_1^2+h_2^2=N(f_v).
\]
Moreover, the common divisor of \(h_1\) and \(h_2\) is exactly \(x\), hence
\[
\operatorname{div}(h_1,h_2)=[0]=\cdiv(f).
\]

Define
\[
f'(x):=f_0(x)+h_1(x)i+h_2(x)j
      =x^3+6+\sqrt{14}\,x\left(x-\frac37\right)i+\frac{5\sqrt{91}}7\,x\,j.
\]
Then \(f'\) takes values in the fixed quaternionic subalgebra
\[
\HH_0=\Span_{\RR}\{1,i,j,k\}\subset \HH.
\]
We now compare the invariants of \(f\) and \(f'\). First,
\[
\Tr(f')=2(x^3+6)=\Tr(f).
\]
Next,
\[
N(f')
=
(x^3+6)^2+h_1^2+h_2^2
=
(x^3+6)^2+N(f_v)
=
N(f).
\]
Finally, since the common divisor of \(h_1\) and \(h_2\) is \(x\), we get
\[
\cdiv(f')=[0]=\cdiv(f).
\]
Therefore \(f\) and \(f'\) are strongly equivalent.

Now set
\[
v:=f'*k.
\]
Using \(ik=-j\) and \(jk=i\), we obtain
\[
v(x)=\frac{5\sqrt{91}}7\,x\,i-\sqrt{14}\,x\left(x-\frac37\right)j+(x^3+6)k.
\]
Thus \(v\) is purely vectorial. Since right multiplication by the nonzero constant
\(k\) preserves the zero set, \(v\) has the same real and spherical zeros as \(f\),
with the same multiplicities, and
\[
N(v)=N(f')=N(f).
\]

By Theorem~\ref{thm:purely-vectorial-model}, the purely vectorial function \(v\)
admits a model \(\widetilde v\). Since \(v\) and \(\widetilde v\) are strongly equivalent,
\(\widetilde v\) has the same real and spherical zeros as \(v\), while all its isolated
nonreal zeros lie in a single slice. Consequently, \(\widetilde v\) is a semi-model for
\(f\).

This example shows concretely that, even when no \(\CC_I\)-preserving representative
exists in the strong-equivalence class, one can still perform the explicit two-direction
reduction globally and then obtain a semi-model by passing to the purely vectorial
function \(v=f'*k\).
\end{example}

\begin{example}\label{ex:genuine-octonionic-semimodel-explicit}
Let \(\Omega=\OO\) and consider the slice regular polynomial
\[
f(x):=x^3+6+x\,e_1+x^2e_2+x^3e_4+x^4e_5.
\]
This is a genuinely octonionic example: \(f\) does not take values in any fixed
quaternionic subalgebra of \(\OO\).

Indeed, assume by contradiction that there exists a quaternionic subalgebra
\(Q\subset \OO\) such that \(f(\OO)\subset Q\). Restricting to real points \(t\in\RR\),
we would have
\[
f(t)=t^3+6+t\,e_1+t^2e_2+t^3e_4+t^4e_5\in Q
\qquad \forall t\in\RR.
\]
Since \(Q\) is a real vector subspace and the polynomials \(1,t,t^2,t^3,t^4\) are linearly
independent over \(\RR\), it would follow that
\[
e_1,e_2,e_4,e_5\in Q,
\]
which is impossible because the imaginary part of a quaternionic subalgebra has real
dimension \(3\).

The scalar part of \(f\) is
\[
f_0(x)=x^3+6,
\]
while the vector part is
\[
f_v(x)=x\,e_1+x^2e_2+x^3e_4+x^4e_5.
\]
Hence
\[
\cdiv(f)=[0],
\]
because the four slice-preserving components have common divisor exactly \(x\), and
\[
N(f_v)=x^2+x^4+x^6+x^8.
\]

Now observe that
\[
x^2+x^4+x^6+x^8
=
(x^4+x)^2+(x^2-x^3)^2.
\]
Therefore we may choose the globally defined slice-preserving polynomials
\[
h_1(x):=x^4+x,
\qquad
h_2(x):=x^2-x^3.
\]
They satisfy
\[
h_1^2+h_2^2=N(f_v),
\]
and their common divisor is exactly \(x\), so
\[
\operatorname{div}(h_1,h_2)=[0]=\cdiv(f).
\]

Fix the quaternionic subalgebra
\[
\HH_0:=\Span_{\RR}\{1,e_1,e_2,e_3\}\subset \OO,
\]
and write \(i:=e_1\), \(j:=e_2\), \(k:=e_3\).
Define
\[
f'(x):=f_0(x)+h_1(x)i+h_2(x)j
      =x^3+6+(x^4+x)i+(x^2-x^3)j.
\]
Then \(f'\) takes values in the fixed quaternionic subalgebra \(\HH_0\).

We now compare the invariants of \(f\) and \(f'\). First,
\[
\Tr(f')=2(x^3+6)=\Tr(f).
\]
Next,
\[
N(f')
=
(x^3+6)^2+h_1^2+h_2^2
=
(x^3+6)^2+N(f_v)
=
N(f).
\]
Finally,
\[
\cdiv(f')=[0]=\cdiv(f),
\]
because the common divisor of \(h_1\) and \(h_2\) is \(x\).

Therefore \(f\) and \(f'\) are strongly equivalent, and Proposition~\ref{prop:two-directions}
is realized globally in this genuinely octonionic case.

Now set
\[
v:=f'*k.
\]
Since \(ik=-j\) and \(jk=i\), we obtain
\[
v(x)=(x^2-x^3)i-(x^4+x)j+(x^3+6)k.
\]
Thus \(v\) is purely vectorial and takes values in \(\HH_0\). Moreover,
\[
N(v)=N(f')=N(f),
\]
and \(v\) has the same real and spherical zeros as \(f\), with the same multiplicities.

By Theorem~\ref{thm:purely-vectorial-model}, the purely vectorial function \(v\)
admits a model \(\widetilde v\). Since \(v\) and \(\widetilde v\) are strongly equivalent,
\(\widetilde v\) has the same real and spherical zeros as \(v\), while all its isolated
nonreal zeros lie in a single slice. Hence \(\widetilde v\) is a semi-model for the
original genuinely octonionic function \(f\).

This example shows that the two-direction reduction and the existence of semi-models
are not merely quaternionic phenomena embedded in \(\OO\), but apply also to functions
whose image is not contained in any quaternionic subalgebra.
\end{example}

\appendix

\section{A local explicit automorphism in Example~\ref{ex:octonionic-model-nonzero-trace}}\label{G2}

The explicit computation in order to get the matrix $\Psi$ appearing in this appendix were assisted by ChatGPT and then independently checked by the authors.

In Example~\ref{ex:octonionic-model-nonzero-trace} we considered the two stem functions
\[
F(z):=z^2+z^4e_1+2z^3e_2+\sqrt6\,z^2e_3+2ze_4+e_5,
\qquad
G(z):=z^2+(z^2+1)^2e_1.
\]
By Theorem~\ref{thm:strong-equivalence-sections}, there exists an intrinsic holomorphic map
\[
\phi:D\longrightarrow \Aut(\OO_\CC)
\]
such that
\[
\phi(z)\bigl(F(z)\bigr)=G(z)
\qquad \forall z\in D.
\]
In this appendix we show that, after restricting to a simply connected open set and choosing
a holomorphic square root, one can write down such an automorphism explicitly.

\subsection*{The construction of a local Cayley frame}

Set
\[
\lambda(z):=(1+z^2)^2,
\qquad
u_1(z):=\frac{1}{\lambda(z)}
\bigl(z^4e_1+2z^3e_2+\sqrt6\,z^2e_3+2ze_4+e_5\bigr).
\]
Since
\[
z^8+4z^6+6z^4+4z^2+1=(1+z^2)^4=\lambda(z)^2,
\]
one has
\[
N\bigl(u_1(z)\bigr)=1.
\]
Fix
\[
u_2:=e_6,
\]
and define
\[
u_3(z):=u_1(z)u_2
=
\frac1{\lambda(z)}
\bigl(2ze_2-e_3-2z^3e_4+\sqrt6\,z^2e_5-z^4e_7\bigr).
\]

Now let \(U\subset D\) be simply connected, and choose a holomorphic branch of
\[
\nu(z):=\sqrt{\,4z^6+6z^4+4z^2+1\,}
\]
on \(U\). Define
\[
u_4(z):=
\frac1{\lambda(z)\nu(z)}
\Bigl(
2z^5e_2-z^4e_3-2z^7e_4+\sqrt6\,z^6e_5
+\bigl(4z^6+6z^4+4z^2+1\bigr)e_7
\Bigr).
\]
Then \((u_1(z),u_2,u_4(z))\) is a local holomorphic Cayley triple. Set
\[
u_5(z):=u_1(z)u_4(z),\qquad
u_6(z):=u_4(z)u_2,\qquad
u_7(z):=u_3(z)u_4(z).
\]
Thus, for each \(z\in U\), the ordered septuple
\[
\bigl(u_1(z),u_2,u_3(z),u_4(z),u_5(z),u_6(z),u_7(z)\bigr)
\]
is a complex Cayley frame of \(\Im(\OO_\CC)\).

Let \(\Psi(z)\) be the \(7\times 7\) matrix whose columns are the coordinates of
\[
u_1(z),u_2,u_3(z),u_4(z),u_5(z),u_6(z),u_7(z)
\]
with respect to the basis \((e_1,\dots,e_7)\). Then
\[
\psi:U\longrightarrow G_2(\CC)\subset \Aut(\OO_\CC)
\]
is holomorphic. By construction,
\[
\psi(z)(e_1)=u_1(z).
\]

\subsection*{The explicit matrix}

Writing
\[
s(z):=\nu(z)=\sqrt{\,4z^6+6z^4+4z^2+1\,},
\]
the matrix of \(\Psi(z)\) in the basis \((e_1,\dots,e_7)\) is
\[
\Psi(z)=
\begin{pmatrix}
\dfrac{z^4}{\lambda} & 0 & 0 & \dfrac{s}{\lambda} & 0 & 0 & 0
\\[1.2ex]
\dfrac{2z^3}{\lambda} & 0 & \dfrac{2z}{\lambda} &
-\dfrac{2z^7}{\lambda s} &
\dfrac{\sqrt6\,z^2}{s} &
\dfrac{2z^5}{\lambda s} &
-\dfrac{1}{s}
\\[1.2ex]
\dfrac{\sqrt6\,z^2}{\lambda} & 0 & -\dfrac{1}{\lambda} &
-\dfrac{\sqrt6\,z^6}{\lambda s} &
-\dfrac{2z^3}{s} &
-\dfrac{z^4}{\lambda s} &
-\dfrac{2z}{s}
\\[1.2ex]
\dfrac{2z}{\lambda} & 0 & -\dfrac{2z^3}{\lambda} &
-\dfrac{2z^5}{\lambda s} &
\dfrac{1}{s} &
-\dfrac{2z^7}{\lambda s} &
\dfrac{\sqrt6\,z^2}{s}
\\[1.2ex]
\dfrac{1}{\lambda} & 0 & \dfrac{\sqrt6\,z^2}{\lambda} &
-\dfrac{z^4}{\lambda s} &
-\dfrac{2z}{s} &
\dfrac{\sqrt6\,z^6}{\lambda s} &
\dfrac{2z^3}{s}
\\[1.2ex]
0 & 1 & 0 & 0 & 0 & 0 & 0
\\[1.2ex]
0 & 0 & -\dfrac{z^4}{\lambda} & 0 & 0 & \dfrac{s}{\lambda} & 0
\end{pmatrix}.
\]
Since \(\Psi(z)\) is an orthogonal change of Cayley frame, one has
\[
\Phi(z):=\Psi(z)^{-1}=\Psi(z)^T .
\]
Therefore \(\Phi(z)\) is explicitly given by
\[
\Phi(z)=
\begin{pmatrix}
\dfrac{z^4}{\lambda} &
\dfrac{2z^3}{\lambda} &
\dfrac{\sqrt6\,z^2}{\lambda} &
\dfrac{2z}{\lambda} &
\dfrac{1}{\lambda} &
0 &
0
\\[1.2ex]
0&0&0&0&0&1&0
\\[1.2ex]
0 &
\dfrac{2z}{\lambda} &
-\dfrac{1}{\lambda} &
-\dfrac{2z^3}{\lambda} &
\dfrac{\sqrt6\,z^2}{\lambda} &
0 &
-\dfrac{z^4}{\lambda}
\\[1.2ex]
\dfrac{s}{\lambda} &
-\dfrac{2z^7}{\lambda s} &
-\dfrac{\sqrt6\,z^6}{\lambda s} &
-\dfrac{2z^5}{\lambda s} &
-\dfrac{z^4}{\lambda s} &
0 &
0
\\[1.2ex]
0 &
\dfrac{\sqrt6\,z^2}{s} &
-\dfrac{2z^3}{s} &
\dfrac{1}{s} &
-\dfrac{2z}{s} &
0 &
0
\\[1.2ex]
0 &
\dfrac{2z^5}{\lambda s} &
-\dfrac{z^4}{\lambda s} &
-\dfrac{2z^7}{\lambda s} &
\dfrac{\sqrt6\,z^6}{\lambda s} &
0 &
\dfrac{s}{\lambda}
\\[1.2ex]
0 &
-\dfrac{1}{s} &
-\dfrac{2z}{s} &
\dfrac{\sqrt6\,z^2}{s} &
\dfrac{2z^3}{s} &
0 &
0
\end{pmatrix}.
\]

\subsection*{Verification that \(\Psi(z)(G(z))=F(z)\)}

Since every algebra automorphism fixes the scalar unit \(1\), one has
\[
\psi(z)\bigl(G(z)\bigr)
=
\Psi(z)\bigl(z^2+\lambda(z)e_1\bigr)
=
z^2+\lambda(z)\Psi(z)(e_1).
\]
By construction,
\[
\Psi(z)(e_1)=u_1(z)
=
\frac{1}{\lambda(z)}
\bigl(z^4e_1+2z^3e_2+\sqrt6\,z^2e_3+2ze_4+e_5\bigr).
\]
Hence
\[
\psi(z)\bigl(G(z)\bigr)
=
z^2+\lambda(z)u_1(z)
=
z^2+z^4e_1+2z^3e_2+\sqrt6\,z^2e_3+2ze_4+e_5
=
F(z).
\]
Therefore
\[
\psi(z)\bigl(G(z)\bigr)=F(z),
\qquad\text{and consequently}\qquad
\phi(z)\bigl(F(z)\bigr)=G(z).
\]

Equivalently, if one writes the vector part of \(F(z)\) as the column vector
\[
\mathbf f(z)=
\begin{pmatrix}
z^4\\ 2z^3\\ \sqrt6\,z^2\\ 2z\\ 1\\ 0\\ 0
\end{pmatrix},
\]
then
\[
\Phi(z)\,\mathbf f(z)
=
\begin{pmatrix}
(1+z^2)^2\\ 0\\ 0\\ 0\\ 0\\ 0\\ 0
\end{pmatrix},
\]
which is precisely the vector part of \(G(z)\).

\begin{remark}
The above construction is local, because it depends on the choice of a holomorphic branch
of
\[
\sqrt{\,4z^6+6z^4+4z^2+1\,}.
\]
Thus the matrix formula for \(\Phi\) is valid on any simply connected open set
\(U\subset D\) where such a branch is fixed and does not vanish. The existence of a
global intrinsic holomorphic automorphism on the whole base is guaranteed abstractly by
Theorem~\ref{thm:strong-equivalence-sections}; the point of this appendix is only to show
that the automorphism can also be written down explicitly, at least locally.
\end{remark}

\section{Direct computation of the zero set described in Example~\ref{ex:no-CI-preserving-representative}}\label{appendix:computation}

In this appendix we verify by direct computation that the quaternions
\[
x_1=\frac85\,i+\frac65\,j,
\qquad
x_2=\frac95\,i+\frac{144}{65}\,j+\frac{12}{13}\,k
\]
are zeros of the polynomial
\(f(x)=x^3-x^2(i+2j+3k)+x(6i-3j+2k)+6\).

\subsection*{The first zero}

Since \(x_1\) is purely imaginary, its square is minus its squared Euclidean norm:
\[
x_1^2
=\left(\frac85 i+\frac65 j\right)^2
=-\left(\frac85\right)^2-\left(\frac65\right)^2
=-\frac{64}{25}-\frac{36}{25}
=-4.
\]
Therefore
\[
x_1^3=x_1x_1^2=-4x_1=-\frac{32}{5}i-\frac{24}{5}j.
\]
Moreover,
\[
-x_1^2(i+2j+3k)=4(i+2j+3k)=4i+8j+12k.
\]
For the linear term we compute
\[
x_1(6i-3j+2k)
=\left(\frac85 i+\frac65 j\right)(6i-3j+2k).
\]
Using the quaternionic multiplication rules, we obtain
\[
\frac85 i\cdot 6i=-\frac{48}{5},
\qquad
\frac85 i\cdot (-3j)=-\frac{24}{5}k,
\qquad
\frac85 i\cdot 2k=-\frac{16}{5}j,
\]
and
\[
\frac65 j\cdot 6i=-\frac{36}{5}k,
\qquad
\frac65 j\cdot (-3j)=\frac{18}{5},
\qquad
\frac65 j\cdot 2k=\frac{12}{5}i.
\]
Hence
\[
x_1(6i-3j+2k)
=-6+\frac{12}{5}i-\frac{16}{5}j-12k.
\]
Putting everything together,
\[
\begin{aligned}
f(x_1)
&=x_1^3-x_1^2(i+2j+3k)+x_1(6i-3j+2k)+6\\
&=\left(-\frac{32}{5}i-\frac{24}{5}j\right)
+(4i+8j+12k)
+\left(-6+\frac{12}{5}i-\frac{16}{5}j-12k\right)+6\\
&=0.
\end{aligned}
\]

\subsection*{The second zero}

Again \(x_2\) is purely imaginary, and its squared norm is
\[
\left(\frac95\right)^2+\left(\frac{144}{65}\right)^2+\left(\frac{12}{13}\right)^2
=\frac{81}{25}+\frac{20736}{4225}+\frac{144}{169}
=9.
\]
Hence
\[
x_2^2=-9,
\qquad
x_2^3=x_2x_2^2=-9x_2
=-\frac{81}{5}i-\frac{1296}{65}j-\frac{108}{13}k.
\]
Also,
\[
-x_2^2(i+2j+3k)=9(i+2j+3k)=9i+18j+27k.
\]
For the linear term we write
\[
x_2(6i-3j+2k)
=\left(\frac95 i+\frac{144}{65}j+\frac{12}{13}k\right)(6i-3j+2k).
\]
Expanding term by term gives
\[
\frac95 i\cdot 6i=-\frac{54}{5},
\qquad
\frac95 i\cdot (-3j)=-\frac{27}{5}k,
\qquad
\frac95 i\cdot 2k=-\frac{18}{5}j,
\]
\[
\frac{144}{65}j\cdot 6i=-\frac{864}{65}k,
\qquad
\frac{144}{65}j\cdot (-3j)=\frac{432}{65},
\qquad
\frac{144}{65}j\cdot 2k=\frac{288}{65}i,
\]
\[
\frac{12}{13}k\cdot 6i=\frac{72}{13}j,
\qquad
\frac{12}{13}k\cdot (-3j)=\frac{36}{13}i,
\qquad
\frac{12}{13}k\cdot 2k=-\frac{24}{13}.
\]
Summing separately the scalar and vector parts, we obtain
\[
-\frac{54}{5}+\frac{432}{65}-\frac{24}{13}=-6,
\]
\[
\frac{288}{65}+\frac{36}{13}
=\frac{288}{65}+\frac{180}{65}
=\frac{468}{65}
=\frac{36}{5},
\]
\[
-\frac{18}{5}+\frac{72}{13}
=-\frac{234}{65}+\frac{360}{65}
=\frac{126}{65},
\]
\[
-\frac{27}{5}-\frac{864}{65}
=-\frac{351}{65}-\frac{864}{65}
=-\frac{1215}{65}
=-\frac{243}{13}.
\]
Therefore
\[
x_2(6i-3j+2k)
=-6+\frac{36}{5}i+\frac{126}{65}j-\frac{243}{13}k.
\]
Consequently,
\[
\begin{aligned}
f(x_2)
&=x_2^3-x_2^2(i+2j+3k)+x_2(6i-3j+2k)+6\\
&=\left(-\frac{81}{5}i-\frac{1296}{65}j-\frac{108}{13}k\right)
+(9i+18j+27k)
+\left(-6+\frac{36}{5}i+\frac{126}{65}j-\frac{243}{13}k\right)+6\\
&=0.
\end{aligned}
\]

We conclude that both \(x_1\) and \(x_2\) are roots of \(f\).

\end{document}